\documentclass[10pt,reqno]{amsart}
\usepackage{amsmath}
\usepackage{amssymb}
\usepackage{amsthm}
\usepackage{eepic,epic}
\usepackage{epsfig}
\usepackage{graphicx}

\newlength{\defbaselineskip}
\newcommand{\setlinespacing}[1]%
           {\setlength{\baselineskip}{#1 \defbaselineskip}}

\numberwithin{equation}{section}

\newtheorem{thm}{Theorem}[section]
\newtheorem{cor}[thm]{Corollary}
\newtheorem{lem}[thm]{Lemma}
\newtheorem{prop}[thm]{Proposition}

\theoremstyle{definition}

\theoremstyle{remark}
\newtheorem{rem}[thm]{Remark}
\numberwithin{equation}{section}

\begin{document}

\title[Strichartz estimates in Wiener amalgam spaces]
{Strichartz estimates in Wiener amalgam spaces and applications to nonlinear wave equations}

\author{Seongyeon Kim, Youngwoo Koh and Ihyeok Seo}

\thanks{This work was supported by NRF-2019R1F1A1054310 (Koh) and NRF-2019R1F1A1061316 (Seo).}


\address{Department of Mathematics, Sungkyunkwan University, Suwon 16419, Republic of Korea}
\email{synkim@skku.edu}

\address{Department of Mathematics Education, Kongju National University, Kongju 32588,
Republic of Korea}
\email{ywkoh@kongju.ac.kr}

\address{Department of Mathematics, Sungkyunkwan University, Suwon 16419, Republic of Korea}
\email{ihseo@skku.edu}

\begin{abstract}
In this paper we obtain some new Strichartz estimates for the wave propagator $e^{it\sqrt{-\Delta}}$
in the context of Wiener amalgam spaces.
While it is well understood for the Schr\"odinger case, nothing is known about the wave propagator.
This is because there is no such thing as an explicit formula for the integral kernel of the propagator
unlike the Schr\"odinger case.
To overcome this lack, we instead approach the kernel by rephrasing it as an oscillatory integral involving Bessel functions
and then by carefully making use of cancellation in such integrals based on the asymptotic expansion of Bessel functions.
Our approach can be applied to the Schr\"odinger case as well.
We also obtain some corresponding retarded estimates to give applications to nonlinear wave equations
where Wiener amalgam spaces as solution spaces can lead to a finer analysis of the local and global behavior of the solution.
\end{abstract}

\maketitle

\section{Introduction} \label{sec1}

The space-time integrability of the wave propagator $e^{it\sqrt{-\Delta}}$, known as \textit{Strichartz estimates},
has been extensively studied over the past several decades and is completely understood:
\begin{equation} \label{classi}
\|e^{it\sqrt{-\Delta}} f\|_{L_t^q L_x^r} \lesssim \|f\|_{\dot H^{\sigma}},
\end{equation}
where $(q,r)$ is wave-admissible, i.e., $2\le q\le \infty$, $2\leq r < \infty$,
\begin{equation} \label{waveadmi}
\frac2q +\frac{n-1}r \le \frac{n-1}2 \quad\text{and}\quad \frac1q + \frac nr= \frac n2 -\sigma.
\end{equation}
The diagonal case $q=r$ was obtained in \cite{Str} in connection with the restriction theorems for the cone.
See \cite{LS,KT} for the general case $q\neq r$.

In this paper we are concerned with obtaining these Strichartz estimates
in Wiener amalgam spaces which, unlike the $L^p$ spaces, control the local regularity of a function and its decay at infinity separately.
This separability makes it possible to perform a finer analysis of the local and global behavior of the propagator.
These aspects were originally pointed out in the several works by Cordero and Nicola \cite{CN,CN2,CN3} in the context of the Schr\"odinger propagator $e^{it\Delta}$,
although the spaces were first introduced by Feichtinger \cite{F}
and have already appeared as a technical tool in the study of partial differential equations (\cite{T}).
See also \cite{S}.

While it is well understood for the Schr\"odinger case, nothing is known about the wave propagator.
The arguments used for the former case take advantage of the explicit formula for the integral kernel of $e^{it\Delta}$.
This makes it possible to obtain some time-deay estimates just by calculating the kernel directly on Wiener amalgam spaces,
and ultimately to appeal to the Keel-Tao approach \cite{KT}.
However, it is no longer available for the wave case in which there is no such thing as an explicit formula for the corresponding kernel.

To overcome this lack, we instead consider the problem of obtaining a pointwise estimate for the integral kernel of the Fourier multiplier ${|\nabla|}^{-\sigma}e^{it\sqrt{-\Delta}}$ by rephrasing it as an oscillatory integral involving Bessel functions using polar coordinates
and then by carefully estimating such integrals on Wiener amalgam spaces based on the asymptotic expansion of Bessel functions.
Our approach in this paper is different from that of Cordero and Nicola mentioned above, and can be applied to the Schr\"odinger case as well (\cite{KKS}).

Before stating our results, we recall the definition of Wiener amalgam spaces.
Let $\varphi \in C_0^{\infty}$ be a test function satisfying $\|\varphi\|_{L^2} =1$.
Let $1\le p,q \le \infty$.
Then the Wiener amalgam space $W(p,q)$
is defined as the space of functions $f \in L_{\textrm{loc}}^p$
equipped with the norm
\begin{equation*}
\|f\|_{W(p,q)}= \big\|\, \|f \tau_x \varphi\|_{L^p} \big\|_{L_x^q},
\end{equation*}
where $\tau_x \varphi(\cdot) = \varphi(\cdot-x)$.
Here different choices of $\varphi$ generate the same space and yield equivalent norms.
This space can be seen as a natural extension of $L^p$ space in view of $W(p,p)=L^p$,
and $W(A,B)$ for Banach spaces $A$ and $B$ is also defined in the same way.
We also list some basic properties of these spaces which will be frequently used in the sequel:
\begin{itemize}
\item
\textit{Inclusion}; if\, $p_0 \ge p_1$ and $q_0 \le q_1$,
\begin{equation} \label{inclusion}
W(p_0, q_0) \subset W(p_1, q_1).
\end{equation}
\item
\textit{Convolution}\footnote{More generally, if\, $A_0 * A_1 \subset A$ and $B_0 * B_1 \subset B$,
$W(A_0 , B_0) * W(A_1 , B_1) \subset W(A, B).$}; if $1/p + 1 = 1/p_0 + 1/p_1$ and $1/q + 1 = 1/q_0 + 1/q_1$,
\begin{equation}\label{y-ineq}
W(p_0, q_0) * W(p_1, q_1) \subset W(p, q).
\end{equation}
\item
\textit{Interpolation}\footnote{For $0<\theta<1$, $(\cdot\,,\cdot)_{[\theta]}$ denotes the complex interpolation functor and $(1/p_\theta,1/q_\theta)$ is usually given as
$1/p_\theta = \theta/p_0 + (1-\theta)/p_1$ and  $1/q_\theta = \theta/q_0 + (1-\theta)/q_1$.};
if\, $q_0 < \infty$ or $q_1 < \infty$,
\begin{equation} \label{inter}
(W(p_0, q_0), W(p_1, q_1))_{[\theta]} = W(p_\theta, q_\theta).
\end{equation}
\item
\textit{Duality}; if $p,q<\infty$,
\begin{equation}\label{dudu}
W(p,q)'=W(p',q').
\end{equation}
\end{itemize}
We refer to \cite{F,F2,F3,H} for details.

\subsection{Strichartz estimates in Wiener amalgam spaces}
Our main results for the Strichartz estimates are now stated as follows.

\begin{thm}\label{thm1}
Let $n\geq 3$ and $\frac n4<\sigma<\frac{n-1}2$.
Assume that $2\le \widetilde q<q<\infty$ and
\begin{equation}\label{ass}
\frac{2n}{n-2\sigma}< r\leq \widetilde r <
\begin{cases}
\frac{4}{n-4\sigma+1}\quad\text{if}\quad\sigma<\frac{n}4+\frac1{2(n-1)},\\
 \frac{2n}{n-2\sigma-1}\quad\text{if}\quad\sigma\geq\frac{n}4+\frac1{2(n-1)}.
\end{cases}
\end{equation}
Then we have
	\begin{equation}\label{T}
	\| e^{it\sqrt{-\Delta}} f\|_{W({\widetilde q} ,q)_t W({\widetilde r} ,r)_x} \lesssim \|f\|_{\dot H^{\sigma}}
	\end{equation}
if
\begin{equation} \label{c1}
	\frac{1}{\widetilde q} + \frac{n-1}{\widetilde r} > \frac{n}{2} -\sigma\quad\text{and}\quad
	\frac{1}{q} + \frac{n}{r} = \frac{n}{2}-\sigma.
	\end{equation}
\end{thm}

\begin{figure} [t]
 \begin{center}
  {\includegraphics[width=0.5\textwidth]{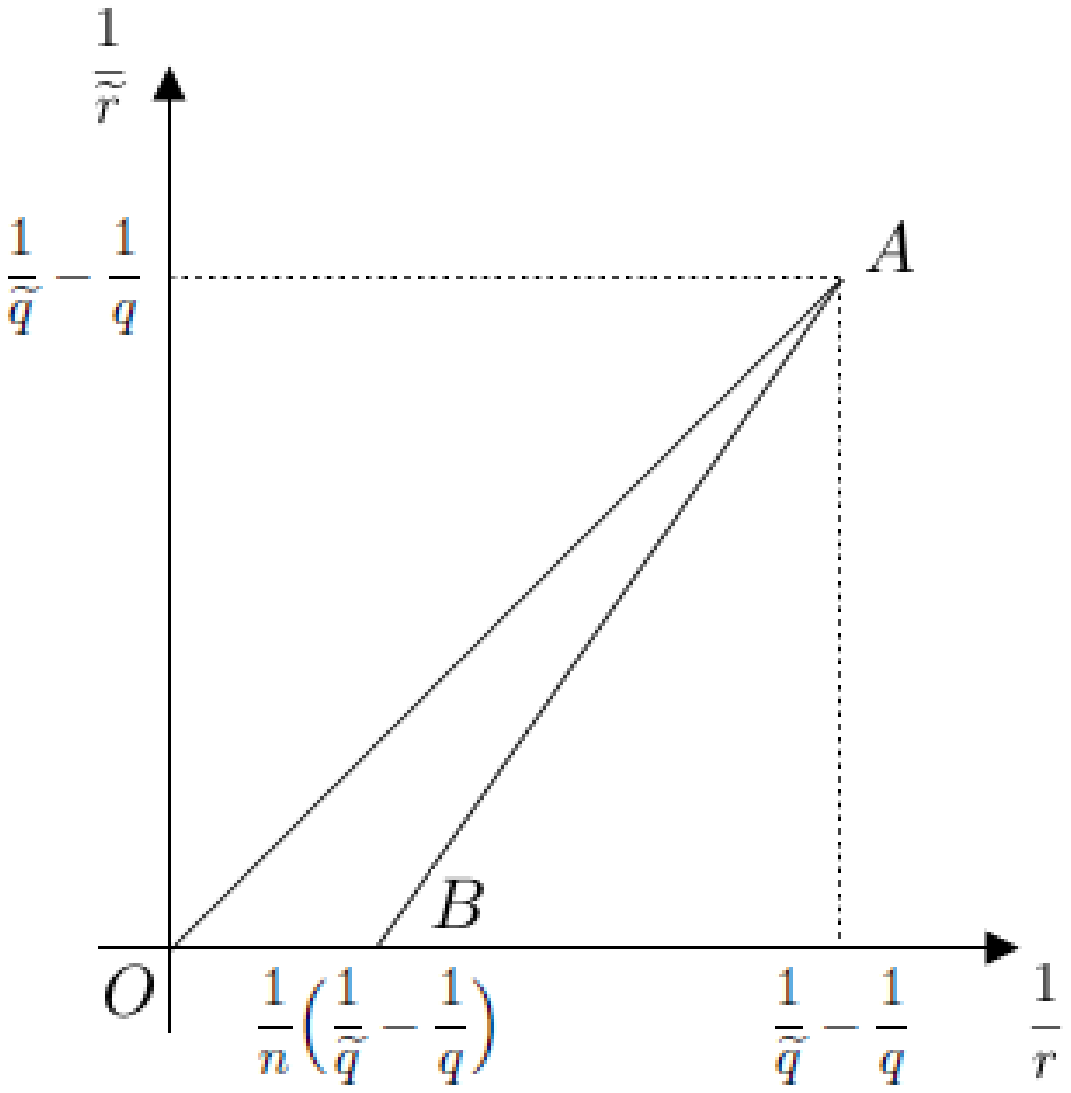}}
 \end{center}
 \caption{\label{figure1}The region of $(1/r,1/\widetilde{r})$ for \eqref{T}}
\end{figure}

\begin{rem}
Roughly speaking, the estimate \eqref{T} shows that the $W(\widetilde{r} ,r)_x$-norm of the propagator has a $L_t^q$-decay at infinity.
But the second condition in \eqref{c1} becomes equivalent to the scaling condition
for the classical estimates \eqref{classi}.
Hence, our estimates are better than the classical ones for large time
since the classical $L_x^r$ norm in \eqref{classi} is rougher than the $W(\widetilde{r} ,r)_x$-norm when $r<\widetilde{r}$ (see \eqref{inclusion}),
although locally the classical $L_t^q$ regularity is replaced by
$L_t^{\widetilde{q}}$ with $\widetilde{q}<q$.
In this regard, it is worth trying to obtain \eqref{T} especially when $r<\widetilde{r}$,
and the theorem shows that $(1/r,1/\widetilde{r})$ can lie in some region inside the triangle with vertices $O,A,B$ in Figure \ref{figure1}.
Note here that \eqref{c1} determines the line through $A$ and $B$, $1/\widetilde{q}-1/q=n/r-(n-1)/\widetilde{r}$.
Some estimates when $r>\widetilde{r}$ are of course  derived trivially from the classical ones \eqref{classi} using the inclusion relation \eqref{inclusion}.
\end{rem}

\begin{rem}
From complex interpolation (see \eqref{inter}) between \eqref{T} and the classical estimates \eqref{classi},
we can easily obtain further estimates on Wiener amalgam spaces extending the range of $\sigma$ to $0\le\sigma< n/2$.
In a different way, one can also trivially increase $q,r$ and diminish $\widetilde q, \widetilde r$ in \eqref{T}
using the inclusion relation \eqref{inclusion}.
\end{rem}

We also obtain the corresponding retarded estimates which are useful to control nonlinearities in relevant nonlinear problems
discussed below.

\begin{thm}\label{thm01}
Let $n\geq 3$ and $\gamma\in(\frac n2,\frac{n+1}2)\cup (\frac{n+1}{2},{n-1})$.
Assume that
\begin{equation*}
	0< \frac{1}{q}+\frac{1}{q_1} < \frac{1}{\widetilde q}+\frac{1}{\widetilde q_1} \leq 1 \quad (q\neq\infty)
\end{equation*}
and
\begin{equation*}
	\max\big\{\frac{n-2\gamma+1}{2},\,\frac{n-\gamma-1}{n}\big\}< \frac{1}{\widetilde r}+\frac{1}{\widetilde r_1} \leq \frac{1}{r}+\frac{1}{r_1} < \frac{n-\gamma}{n}.
\end{equation*}
Then we have
\begin{equation} \label{reta}
	\bigg\|\int_{0}^t e^{i(t-s)\sqrt{-\Delta}}|\nabla|^{-\gamma}F(\cdot,s)ds \bigg\|_{W(\widetilde q, q)_tW(\widetilde r , r)_x} \lesssim \|F\|_{W({\widetilde q}_1', q_1')_tW(\widetilde r_1' , r_1')_x}
\end{equation}
if
\begin{equation} \label{ic}
	\frac{1}{\widetilde q} +\frac{1}{\widetilde q_1} + \frac{n-1}{\widetilde r} +\frac{n-1}{\widetilde r_1} > n -\gamma \quad\text{and}\quad
	\frac{1}{q} +\frac{1}{q_1} + \frac{n}{r} + \frac{n}{r_1} = n-\gamma.
\end{equation}
\end{thm}

\begin{rem}
It is well known that some retarded estimates can be derived from the homogeneous estimates using
the $TT^\ast$ argument and the Christ-Kiselev lemma \cite{CK}.
But this standard method is not accessible in the context of Wiener amalgam spaces because of lack of the corresponding lemma,
and therefore we need to approach the matter more directly.
\end{rem}

\subsection{Application to nonlinear wave equations}
Now we turn to a few applications of our estimates to local well-posedness of the Cauchy problem for nonlinear wave equations
\begin{equation} \label{eq}
	\begin{cases}
		\partial_t^2 u - \Delta u = F_k(u), \\
		u(x,0)=f(x)\in\dot{H}^\sigma,\\
\partial_tu(x,0)=g(x)\in\dot{H}^{\sigma-1},
	\end{cases}
\end{equation}
where $(x,t)\in\mathbb{R}^n\times\mathbb{R}$
and the nonlinearity $F_k \in C^1(\mathbb{R})$ ($k>1$) satisfies
\begin{equation} \label{cF}
	|F_k(u)| \lesssim |u|^k \quad\text{and}\quad |u||F_k'(u)| \sim |F_k(u)|.
\end{equation}
The typical models are when $F_k(u) =\pm u^k$ or $\pm|u|^{k-1}u$.

The problem of determining the largest $k$ for which \eqref{eq} is locally well-posed was addressed for
higher dimensions $n\geq3$ in \cite{K}, and then almost completely answered (\cite{L,LS,KT});
the problem \eqref{eq} is ill-posed when
\begin{equation}\label{larg}
k>k(\sigma)=
\begin{cases}
1+\frac4{n+1-4\sigma}\quad\text{if}\quad\sigma\le1/2,\\
1+\frac4{n-2\sigma}\quad\text{if}\quad\sigma\ge1/2.
\end{cases}
\end{equation}
The endpoint case $\sigma=0$ ($k=2$) when $n=3$ is also ill-posed.
When $\sigma\geq\frac{n-3}{2(n-1)}$, it is known that $k(\sigma)$ given by \eqref{larg} is indeed best-possible,
but for $\sigma<\frac{n-3}{2(n-1)}(<1/2)$ the sharpness is not yet known.
The piecewise smooth curve in Figure \ref{figure} describes the maximal $k$ particularly when $n=3$.

\begin{figure} [t]
	\begin{center}
		{\includegraphics[width=0.5\textwidth]{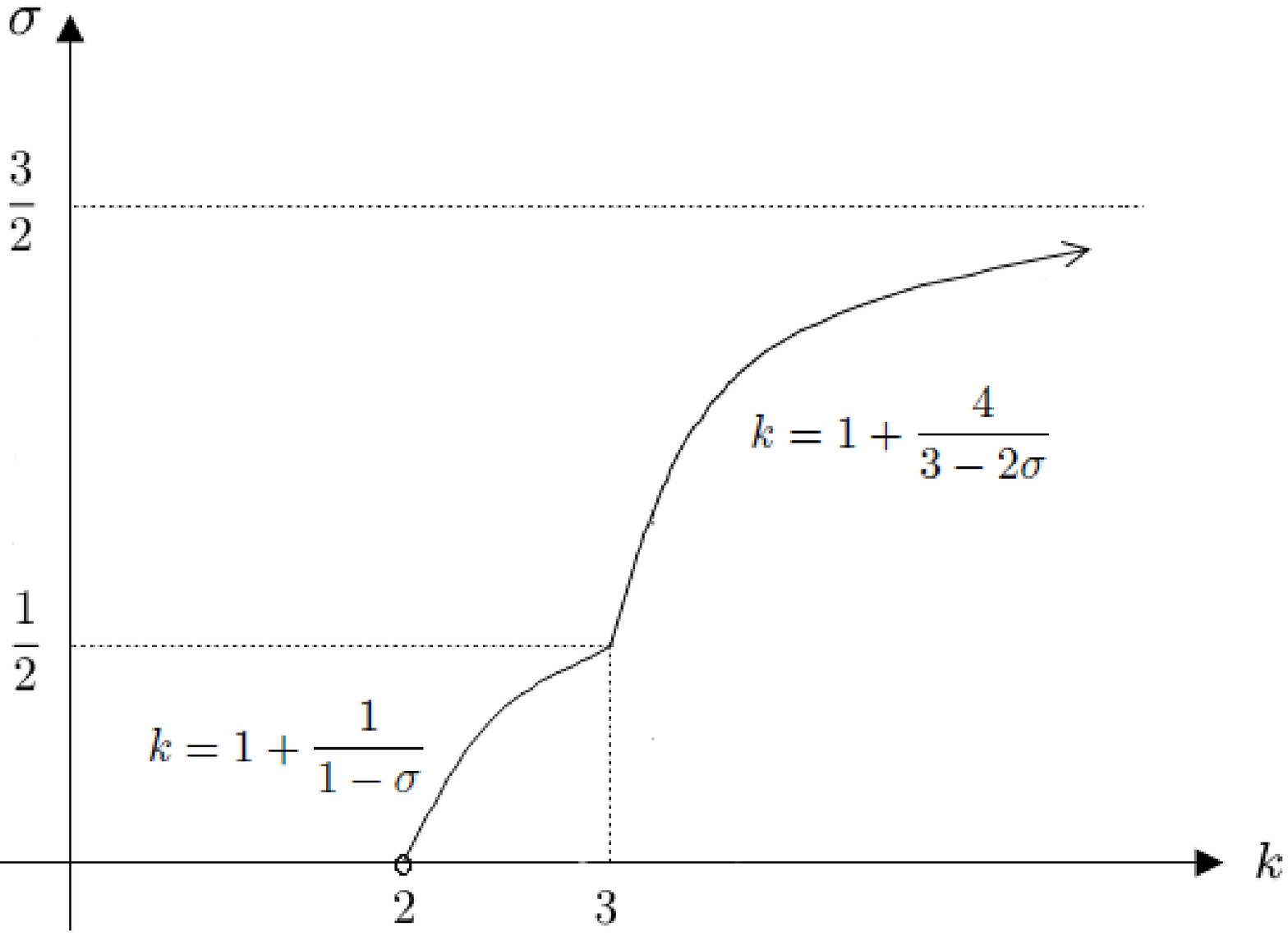}}
	\end{center}	
	\caption{\label{figure}The maximal $k(\sigma)$ for $n=3$}
\end{figure}

The second aim in this paper is to study this nonlinear problem in the context of Wiener amalgam spaces
by making use of our estimates.
The motivation behind this is that
these spaces as solution spaces can control the local regularity of the solution and its decay at infinity separately.
This can lead to a finer analysis of the local and global behavior of the solution.
Since the case $\sigma\geq1/2$ was settled for all dimensions $n\ge3$,
we shall only discuss in detail the low-regularity case $\sigma\le1/2$
although the former case can be also handled in the same way.
We shall also restrict ourselves to three physical dimension $n=3$ just for brevity
since the proof of our existence results for higher dimensions $n\geq4$ follows similar lines.

\begin{thm}\label{thmthm}
  Let $0<\sigma\le1/2$ and $1 < k <k(\sigma)$.
  Then there exist a time $T>0$ and a unique solution $u$ to \eqref{eq}
  with
  \begin{equation}\label{qaz0}
  u \in W(\widetilde q , q)_t ([0,T];W(\widetilde r, r)_x (\mathbb{R}^3))
  \end{equation}
  where
  \begin{equation}\label{qaz2}
\begin{gathered}
0<\frac{1}{q} < \frac{1}{\widetilde q} \leq \frac{1}{4},
\quad \frac{1}{4} < \frac{1}{\widetilde r} \le \frac{1}{r}<\frac{1}{2},\\
1\leq \frac{1}{q}+\frac{3}{r} < \min\big\{\frac32-\frac2{\widetilde{q}},\,\frac2{\widetilde{q}}+\frac4{\widetilde{r}}-\frac12\big\}.
\end{gathered}
\end{equation}
Furthermore, the solution satisfies
  \begin{equation}\label{data}
   u\in C_t([0,T];\dot{H}^{\sigma})\cap C_t^1([0,T];\dot{H}^{\sigma-1})
  \end{equation}
and depends continuously (in the norms \eqref{qaz0},\eqref{data}) on the data.
  \end{thm}

From the proof one can give a precise estimate for the life span of the solution according to the size of the initial data as
$$ T \sim (\|f\|_{\dot H^{\sigma}}+\|g\|_{\dot H^{\sigma-1}})^{-c(k-1)}$$
for some $c=c(q,\widetilde{q},r,\widetilde{r},k)>0$ (see \eqref{M}).
Once the local existence is shown for small $T$, it holds therefore for any finite time $T$ by a straightforward iteration using
$$\sup_{t\in[0,T]}\|u\|_{\dot{H}^\sigma}+\sup_{t\in[0,T]}\|u\|_{\dot{H}^{\sigma-1}}\lesssim \|f\|_{\dot{H}^\sigma}+\|g\|_{\dot{H}^{\sigma-1}}$$
(see \eqref{CC1},\eqref{CC2}).

\

\noindent\textit{Outline of the paper.}
In Section \ref{sec2} we prove Theorem \ref{thm1}; using the $TT^\ast$ argument we first rephrase \eqref{T} as
\begin{equation*}
\bigg\| \int_{\mathbb{R}} (K_{2\sigma}(\cdot,t-s) * F(\cdot,s))(x) ds \bigg\|_{W(\widetilde q ,q)_t W(\widetilde r,r)_x} \lesssim \|F\|_{W(\widetilde q' ,q')_t W(\widetilde r',r')_x}
\end{equation*}
where $K_{\gamma}$ denotes the integral kernel of the Fourier multiplier $|\nabla|^{-\gamma}e^{it\sqrt{-\Delta}}$ given as in \eqref{ker},
and then make use of the convolution relation to reduce the matter to
$
\| K_{2\sigma} \|_{W(L^{\frac{\widetilde q}{2}}, L^{\frac q2,\infty})_tW(\frac{\widetilde r}{2},\frac{r}{2})_x} < \infty.
$
This is carried out by estimating the time-decay estimates
$\| K_{2\sigma}(\cdot,t) \|_{W(\frac{\widetilde r}{2},\frac{r}{2})_x}\lesssim|t|^{-\omega}$
with suitable $\omega(r,\widetilde{r},\sigma,n)>0$ to insure $|t|^{-\omega}\in W(L^{\frac{\widetilde q}{2}}, L^{\frac q2,\infty})_t$
under the second condition in \eqref{c1}.
To obtain such decay estimates (Proposition \ref{fix}),
we first rephrase the kernel in Section \ref{sec3} as an oscillatory integral like \eqref{K} involving Bessel functions
to carefully estimate it making use of cancellation in such integrals based on the asymptotic expansion of Bessel functions,
and then use these estimations (Proposition \ref{pw}) to calculate the time-decay bounds in Section \ref{sec4}.
Section \ref{sec5} is devoted to proving the corresponding retarded estimates in Theorem \ref{thm01},
and finally we apply the Strichartz estimates to the nonlinear problem \eqref{eq}
to prove Theorem \ref{thmthm} in the final section, Section \ref{sec6}.

\

Throughout this paper, the letter $C$ stands for a positive constant which may be different
at each occurrence.
We also denote $A\lesssim B$ to mean $A\leq CB$ with unspecified constants $C>0$.

\section{Homogeneous estimates} \label{sec2}
In this section we prove Theorem \ref{thm1}.
To prove \eqref{T}, we can apply the standard TT* argument because of the H\"older's type inequality
\begin{equation*}
|\langle F,G \rangle_{L^2_{x,t}}|
\leq \|F\|_{W(\widetilde q, q)_t W(\widetilde r, r)_x} \|G\|_{W({\widetilde q}' , q')_t W({\widetilde r}', r')_x}
\end{equation*}
which can be proved directly from the definition of Wiener amalgam spaces.
Hence it is enough to show that
\begin{equation}\label{TT*}
\bigg\| \int_\mathbb{R} |\nabla|^{-2\sigma} e^{i(t-s)\sqrt{-\Delta}}  F(\cdot,s) ds \bigg\|_{W(\widetilde q , q)_t W(\widetilde r , r)_x}
\lesssim \|F\|_{W(\widetilde q' , q')_t W(\widetilde r' , r')_x}.
\end{equation}
We first write the integral kernel of the Fourier multiplier $|\nabla|^{-\gamma}e^{it\sqrt{-\Delta}}$ as
\begin{equation} \label{ker}
K_{\gamma} (x,t) :=\frac{1}{(2\pi)^n}\int_{\mathbb{R}^n} e^{i(x \cdot \xi + t|\xi|)} \frac{d\xi}{|\xi|^{\gamma}}.
\end{equation}
Then \eqref{TT*} is rephrased as follows:
\begin{equation}\label{conv}
\bigg\| \int_{\mathbb{R}} (K_{2\sigma}(\cdot,t-s) * F(\cdot,s))(x) ds \bigg\|_{W(\widetilde q ,q)_t W(\widetilde r,r)_x} \lesssim \|F\|_{W(\widetilde q' ,q')_t W(\widetilde r',r')_x}.
\end{equation}

Now we will prove \eqref{conv}.
By Minkowski's inequality and the convolution relation \eqref{y-ineq},
it follows that
\begin{align}\label{bfHLS}
\nonumber \bigg\| \int_{\mathbb{R}} (K_{2\sigma} (\cdot,t-s)* &F(\cdot,s))(x) ds \bigg\|_{W(\widetilde q ,q)_t W(\widetilde r ,r)_x} \\
\nonumber
&\le \bigg\| \int_{\mathbb{R}} \|K_{2\sigma}(\cdot,t-s) * F(\cdot,s)\|_{W(\widetilde r ,r)_x} ds \bigg\|_{W(\widetilde q,q)_t} \\
&\le \bigg\| \int_{\mathbb{R}} \|K_{2\sigma}(\cdot,t-s)\|_{W(\frac{\widetilde r}{2} ,\frac{r}{2})_x} \|F(\cdot,s)\|_{W(\widetilde r',r')_x} ds\bigg\|_{W(\widetilde q ,q)_t}.
\end{align}
Recall the Hardy-Littlewood-Sobolev fractional integration theorem (see e.g. \cite{St}, p. 119) in dimension $1$:
\begin{equation}\label{hls}
L^{\frac{1}{\alpha},\infty}\ast L^p\hookrightarrow L^q
\end{equation}
for $0<\alpha < 1$ and $1\le p<q<\infty$ with $\frac{1}{q} +1 = \frac{1}{p} + \alpha$.
Using \eqref{hls} with $\alpha = \frac{2}{q}$ and $p=q'$
and the usual Young's inequality, the convolution relation again gives
\begin{equation*}
W(L^{\frac{\widetilde q}2}, L^{\frac q2, \infty})_t * W(\widetilde q', q')_t \subset W(\widetilde q, q )_t
\end{equation*}
for $2\leq\widetilde{q}\leq\infty$ and $2<q<\infty$.
Hence we get
\begin{align}\label{bfHLS2}
\nonumber\bigg\| \int_{\mathbb{R}} \|K_{2\sigma}(\cdot,t-s)\|_{W(\frac{\widetilde r}{2} ,\frac{r}{2})_x} &\|F(\cdot,s)\|_{W(\widetilde r' ,r')_x} ds\bigg\|_{W(\widetilde q ,q)_t}\\
\lesssim &
\| K_{2\sigma} \|_{{W(L^{\frac{\widetilde q}{2}} ,  L^{\frac q2,\infty})_t}{W(\frac{\widetilde r}{2} ,\frac{r}{2})_x}}
\|F\|_{W(\widetilde q' ,q')_t W(\widetilde r',r')_x}.
\end{align}
Combining \eqref{bfHLS} and \eqref{bfHLS2}, we now obtain the desired estimate \eqref{conv} if
\begin{equation}\label{t}
\| K_{2\sigma} \|_{{W(L^{\frac{\widetilde q}{2}}, L^{\frac q2,\infty})_t}{W(\frac{\widetilde r}{2},\frac{r}{2})_x}} < \infty
\end{equation}
for $(\widetilde q , \widetilde r)$ and $(q,r)$ satisfying the same conditions as in Theorem \ref{thm1}.

To show \eqref{t}, we use the following time-decay estimates for the integral kernel \eqref{ker}
which will be obtained in later sections.

\begin{prop} \label{fix}
Let $n\ge 3$ and $\gamma\in(\frac n2,\frac{n+1}2)\cup(\frac{n+1}2,n-1)$.
Assume that
\begin{equation} \label{ftc}
\frac{2n}{n-\gamma}<r\le \widetilde r <
\begin{cases}
\frac{4}{n-2\gamma+1}\quad\text{if}\quad\gamma<\frac n2+\frac1{n-1},\\
\frac{2n}{n-\gamma-1}\quad\text{if}\quad\gamma\ge\frac n2+\frac1{n-1}.
\end{cases}
\end{equation}
Then we have
	\begin{equation}\label{x}
	\|K_{\gamma}(\cdot,t)\|_{W(\frac{\widetilde r}{2} ,\frac{r}{2})_x} \lesssim
	\begin{cases}
	|t|^{-n+\gamma + {2(n-1)}/{\widetilde r}} \quad\text{if} \quad |t| \le 1, \\
	|t|^{-n+\gamma + {2n}/{r}} \quad \text{if} \quad |t|  \ge1.
	\end{cases}
	\end{equation}	
\end{prop}

To begin with, we set $h(t)= \|K_{2\sigma}(\cdot,t)\|_{W(\frac{\widetilde r}{2} ,\frac{r}{2})_x}$
and choose $\varphi(t) \in C_0^{\infty}(\mathbb{R})$ supported on $\{t\in\mathbb{R}:|t|\le1\}$.
To calculate $\| h \|_{{W(L^{\frac{\widetilde q}2}, L^{\frac{q}{2},\infty})_t}}$ using \eqref{x},
we handle $\|h \tau_k \varphi \|_{L_t^{\widetilde q/2}}$ dividing cases into $|k|\le 1$, $1\le|k|\le 2$ and $|k|\ge 2$.
First we consider the case $|k|\le 1$. By using \eqref{x} with $\gamma=2\sigma$ and the support condition of $\varphi$,
\begin{equation}\label{dfg}
\|h \tau_k \varphi \|_{L_t^{\tilde q/2}}^{\widetilde q/2}
\lesssim \int_{0<|t|\le1} |t|^{\frac{\widetilde q}{2}(-n+2\sigma + \frac{2(n-1)}{\widetilde r})} dt + \int_{1\le|t|\le|k|+1} |t|^{\frac{\widetilde q}{2}(-n+2\sigma + \frac{2n}{r})} dt.
\end{equation}
Since $\frac{\widetilde q}{2} (-n+2\sigma + \frac{2(n-1)}{\widetilde r})+1>0$ by the first condition in \eqref{c1},
the first integral in \eqref{dfg} is trivially finite.
The second integral is bounded as follows:
\begin{align}\label{ert}
\nonumber\int_{1\le|t|\le|k|+1} |t|^{\frac{\widetilde q}{2}(-n+2\sigma + \frac{2n}{r})} dt
&\lesssim\frac{(|k|+1)^{\frac{\widetilde q}{2}(-n+2\sigma + \frac{2n}{r})+1} - 1 } {\frac{\widetilde q}{2}(-n+2\sigma + \frac{2n}{r})+1}\\
&\lesssim |k|.
\end{align}
Indeed, since ${\frac{\widetilde q}{2}(-n+2\sigma + \frac{2n}{r})}<0$
by the first inequality in the condition \eqref{ass}, the second inequality in \eqref{ert} follows easily from the mean value theorem.
Hence we get
\begin{equation*}
\|h \tau_k \varphi \|_{L_t^{\widetilde q/2}}^{\widetilde q/2} \lesssim 1
\end{equation*}
when $|k|\le 1$.
The other cases $1\le|k|\le 2$ and $|k|\ge 2$ are handled in the same way:

\begin{align}
\nonumber
\|h \tau_k \varphi \|_{L_t^{\widetilde q/2}}^{\widetilde q/2}
&\lesssim \int_{|k|-1\le|t|\le1} |t|^{\frac{\widetilde q}{2}(-n+2\sigma + \frac{2(n-1)}{\widetilde r})} dt + \int_{1\le|t|\le|k|+1} |t|^{\frac{\widetilde q}{2}(-n+2\sigma + \frac{2n}{r})} dt \\
\nonumber
&\lesssim 1+ |k| \\
\nonumber
&\lesssim 1
\end{align}
when $1\le|k|\le 2$, and when $|k|\ge 2$

\begin{align}
\nonumber
\|h \tau_k \varphi \|_{L_t^{\widetilde q/2}}^{\widetilde q/2}
&\lesssim \int_{{|k|-1}\le|t|\le{|k|+1}} |t|^{\frac{\widetilde q}{2}(-n+2\sigma + \frac{2n}{r})} dt \\
\nonumber
&\lesssim\frac{(|k|+1)^{\frac{\widetilde q}{2}(-n+2\sigma + \frac{2n}{r})+1} - (|k|-1)^{\frac{\widetilde q}{2}(-n+2\sigma + \frac{2n}{r})+1}}{\frac{\widetilde q}{2}(-n+2\sigma + \frac{2n}{r})+1} \\
\nonumber
&\lesssim (|k|-1)^{\frac{\widetilde q}{2}(-n+2\sigma + \frac{2n}{r})}.
\end{align}
Consequently, we get

\begin{equation} \label{local_t}
\|h\tau_k\varphi\|_{L_t^{\widetilde q/2}} \lesssim
\begin{cases}
1 \quad\textit{if}\quad |k| \leq 2,\\
(|k|-1)^{-(n-2\sigma - \frac{2n}{r})} \quad\textit{if}\quad |k| \geq 2.
\end{cases}
\end{equation}
By \eqref{local_t}, $\|h\tau_k\varphi\|_{L_t^{\widetilde q/2}}$  belongs to $L^{\frac{q}{2}, \infty}_k$
since we are assuming the second condition in \eqref{c1} which is equivalent to $\frac{2}{q} = n-2\sigma-\frac{2n}{r}$.
This finally implies
\begin{equation*}
\| h \|_{{W(L^{\frac{\widetilde q}{2}}, L^{\frac{q}{2},\infty})_t}}<\infty
\end{equation*}
for $(\widetilde q , \widetilde r)$ and $(q,r)$ satisfying the same conditions as in Theorem \ref{thm1}
except for the case $\sigma=(n+1)/4$ when $n\geq4$.
But this case can be shown by just interpolating the cases $\sigma=(n+1)/4-\varepsilon$ and $\sigma=(n+1)/4+\varepsilon$
with a sufficiently small $\varepsilon>0$.
Note finally that the condition $q>\widetilde{q}$ in the theorem follows immediately from combining \eqref{c1} and $r\le\widetilde{r}$.

\section{Pointwise estimates for the integral kernel} \label{sec3}
In this section we first obtain the following pointwise estimates for the integral kernel \eqref{ker}
which will be used in the next section to finish the proof of the time-decay estimates (Proposition \ref{fix}):

\begin{prop} \label{pw}
	Let $n\geq3$.
If $n\geq4$ and $(n+1)/2<\gamma<n-1$,
\begin{equation}\label{kerbd0}
|K_{\gamma}(x,t)|\lesssim
\begin{cases}
|x|^{-(n-\gamma)}\quad\text{if}\quad|x|\geq|t|/2,\\
|t|^{-1}|x|^{-(n-1-\gamma)}\quad\text{if}\quad\quad|x|\leq|t|/2,
\end{cases}
\end{equation}
and if $(n-1)/2< \gamma < (n+1)/2$
	\begin{equation}\label{kerbd}
	|K_{\gamma} (x,t) | \lesssim
	\begin{cases}
|x|^{-\frac{n-1}{2}} \big||x|-|t|\big|^{-(\frac{n+1}{2}-\gamma)} \quad {\text{if}} \quad |x| \geq |t|/2,\\
|t|^{-1}|x|^{-(n-1-\gamma)} \quad {\text{if}} \quad |x| \leq |t|/2.
	\end{cases}
	\end{equation}
(If $1/2<\gamma<1$, \eqref{kerbd} is still true for $n=2$).
\end{prop}
\begin{proof}[Proof]
Using polar coordinates $\xi= \omega\xi'$ and $x=rx'$ where $\xi',x' \in \mathbb{S}^{n-1}$, $\omega = |\xi|$ and $r=|x|$,
we first write
\begin{align} \label{K}
\nonumber
K_{\gamma}(x,t) &= \frac{1}{(2\pi)^n} \int_{0}^{\infty} e^{it\omega} \omega^{n-\gamma-1}  \int_{\mathbb{S}^{n-1}} e^{ir\omega x' \cdot \xi'} d\xi' d\omega \\
&= C_n r^{-\frac{n-2}{2}} \int_0^{\infty} e^{it\omega} \omega^{\frac{n}{2}- \gamma}J_{\frac{n-2}{2}} (r\omega) d\omega.
\end{align}
Here we also used the fact (see, for example, \cite{G}, p. 428) that
\begin{equation*}
\int_{\mathbb{S}^{n-1}} e^{ir\omega x' \cdot \xi'} d\xi'
=C_n (r\omega)^{-\frac{n-2}{2}}J_{\frac{n-2}{2}} (r\omega)
\end{equation*}
where $J_\nu$ denotes the Bessel function of complex order $\nu$ with $\text{Re}\,\nu>-1/2$.

\subsection{The case $\frac{n+1}2<\gamma<n-1$}
In this case we will obtain
\begin{equation*}
|K_{\gamma}(x,t)|\lesssim \min\big\{|x|^{-n+\gamma},|t|^{-1}|x|^{-n+\gamma+1}\big\}
\end{equation*}
which implies \eqref{kerbd0}.
We first show $|K_{\gamma}(x,t)|\lesssim |x|^{-n+\gamma}$.
Motivated by the following fact (see, for example, \cite{G}, Appendix B) that for $\text{Re}\,\nu> -1/2$
	\begin{equation}\label{Bfbd}
	|J_{\nu}(m) | \le
	\begin{cases}
	C_{\nu} m^{\text{Re}\nu} \quad \text{if} \quad 0 < m < 1, \\
	C_{\nu} m^{-\frac{1}{2}} \quad \text{if} \quad m\geq 1,
	\end{cases}
	\end{equation}
we first split the integral in \eqref{K} into the regions $r\omega<1$ and $r\omega>1$:
\begin{equation}\label{C1}
\int_0^{1/r} e^{it\omega} \omega^{\frac{n}{2}-\gamma} J_{\frac{n-2}{2}}(r\omega) d\omega
+\int_{1/r}^{\infty} e^{it\omega} \omega^{\frac{n}{2}-\gamma} J_{\frac{n-2}{2}}(r\omega) d\omega.
\end{equation}
Using \eqref{Bfbd}, the first part is then bounded as
\begin{align}
\nonumber
\bigg|\int_0^{1/r} e^{it\omega} \omega^{\frac{n}{2}-\gamma} J_{\frac{n-2}{2}}(r\omega) d\omega\bigg|
&\lesssim \int_0^{1/r}\omega^{\frac{n}{2}-\gamma} \big|J_{\frac{n-2}{2}}(r\omega)\big| d\omega \\
\nonumber
&\lesssim \int_{0}^{1/r} \omega^{\frac{n}{2}-\gamma} (r\omega)^{\frac{n-2}{2}} d\omega \\
\label{I}
&\lesssim r^{\frac{n-2}{2}}r^{-n+\gamma}
\end{align}
when $\gamma<n$.
A similar argument gives the same bound for the second part under the condition $\gamma>(n+1)/2$, and thus
\begin{equation*}
|K_{\gamma}(x,t)|\lesssim |x|^{-n+\gamma},
\end{equation*}
if $\frac{n+1}2<\gamma<n$ (which is wider than what we want in the first place).

To show $|K_{\gamma}(x,t)|\lesssim |t|^{-1}|x|^{-n+\gamma+1}$ this time,
we start with applying integration by parts to the integral in \eqref{K} as
	\begin{equation} \label{ibp}
	\int_0^{\infty} e^{it\omega} \omega^{\frac{n}{2}- \gamma}J_{\frac{n-2}{2}} (r\omega) d\omega
=- \int_0^{\infty} \frac{e^{it \omega}}{it} \frac{d}{d\omega} \big(\omega^{\frac{n}{2}-\gamma}J_{\frac{n-2}{2}}  (r\omega) \big) d\omega,
	\end{equation}
where the boundary terms vanish when $\frac{n-1}2<\gamma<n-1$; using \eqref{Bfbd},
\begin{equation*}
\bigg| \frac{e^{it\omega}}{it} r^{-\frac{n-2}{2}} \omega^{\frac{n}{2}-\gamma} J_{\frac{n-2}{2}}(r\omega)\bigg|
\lesssim\begin{cases}
|t|^{-1} \omega^{n-1-\gamma} \rightarrow 0 \quad \text{as} \quad \omega \rightarrow 0,\\
|t|^{-1} r^{-\frac{n-1}{2}}\omega^{\frac{n-1}{2}-\gamma} \rightarrow 0 \quad \text{as} \quad \omega \rightarrow \infty.
\end{cases}
\end{equation*}
We then use the following property (see, for example, \cite{G}, p. 425)
	\begin{equation*}
	\frac{d}{dm}(m^{-\nu} J_\nu(m)) = -m^{-\nu} J_{\nu+1}(m) \quad \text{for}\quad\text{Re}\, \nu >-1/2,
	\end{equation*}
to estimate the derivative term in \eqref{ibp} as
\begin{align*}
 \frac{d}{d\omega} \bigg(\omega^{\frac{n}{2}-\gamma} r^{-\frac{n-2}{2}} J_{\frac{n-2}{2}} (r\omega)\bigg)
	&= \frac{d}{d\omega} \bigg( (r\omega)^{-\frac{n-2}{2}} J_{\frac{n-2}{2}} (r\omega)\cdot \omega^{n-\gamma-1} \bigg) \\
	&= Cr^{-\frac{n-2}{2}} \omega^{\frac{n-2}{2}-\gamma}  J_{\frac{n-2}{2}}(r\omega)
- r^{-\frac{n-4}{2}}  \omega^{\frac{n}{2}-\gamma}  J_{\frac{n}{2}}(r\omega).
\end{align*}
Hence we arrive at
\begin{equation}\label{ibp2}
\begin{aligned}
	r^{-\frac{n-2}2}\int_0^{\infty} e^{it\omega} \omega^{\frac{n}{2}- \gamma}J_{\frac{n-2}{2}} (r\omega) d\omega
=&-C r^{-\frac{n-2}{2}} \int_0^{\infty} \frac{e^{it \omega}}{it} \omega^{\frac{n-2}{2}-\gamma}  J_{\frac{n-2}{2}}(r\omega) d\omega\\
&+ r^{-\frac{n-4}{2}} \int_0^{\infty} \frac{e^{it \omega}}{it} \omega^{\frac{n}{2}-\gamma}  J_{\frac{n}{2}}(r\omega) d\omega.
\end{aligned}
\end{equation}
By splitting the integrals in the right side of \eqref{ibp2} into the regions $r\omega<1$ and $r\omega>1$,
and using \eqref{Bfbd} as above, one can see that
\begin{equation}\label{osci2}
\bigg|\int_0^{\infty} \frac{e^{it \omega}}{it} \omega^{\frac{n-2}{2}-\gamma}  J_{\frac{n-2}{2}}(r\omega) d\omega\bigg|
\lesssim |t|^{-1}r^{-n/2+\gamma}
\end{equation}
if $\frac{n-1}2<\gamma<n-1$,
and if $\frac{n+1}2<\gamma<n+1$
\begin{equation*}
\bigg|\int_0^{\infty} \frac{e^{it \omega}}{it} \omega^{\frac{n}{2}-\gamma}  J_{\frac{n}{2}}(r\omega) d\omega\bigg|
\lesssim |t|^{-1}r^{-n/2+\gamma-1}.
\end{equation*}
Combining these estimates, we therefore get for $\frac{n+1}2<\gamma<n-1$
$$\bigg|\int_0^{\infty} e^{it\omega} \omega^{\frac{n}{2}- \gamma}J_{\frac{n-2}{2}} (r\omega) d\omega\bigg|
\lesssim  |t|^{-1}r^{\frac{n-2}2}r^{-n+\gamma+1}$$
which implies immediately
\begin{equation*}
|K_{\gamma}(x,t)|\lesssim |t|^{-1}|x|^{-n+\gamma+1}
\end{equation*}
from \eqref{K}.

\subsection{The case $\frac{n-1}2<\gamma<\frac{n+1}2$ ($1/2<\gamma<1$ if $n=2$)}

In order to obtain this case,
we shall make use of the following asymptotic expansion of Bessel functions (see, for example, \cite{G}, Appendix B)
for the region $r\omega>1$ in the previous argument, and use cancellation in $e^{i(t\pm r)\omega}$.

\begin{lem} \label{asyBess}
	For $m>1$ and $\text{Re}\,\nu >-1/2$,
	\begin{equation} \label{asyBe}
	J_{\nu}(m) = \frac{1}{\sqrt{2\pi m}}\big( e^{i( m - \frac{\pi \nu}{2} -\frac{\pi}{4})}+ e^{-i( m - \frac{\pi \nu}{2} -\frac{\pi}{4})}\big)
+ R_{\nu} (m)
	\end{equation}
	where
\begin{equation}\label{err}
|R_{\nu} (m)| \leq C_{\nu} m^{-3/2}.
\end{equation}
\end{lem}

Inserting \eqref{asyBe} into the second integral in \eqref{C1}, we now see that
\begin{align}\label{cance}
\nonumber\int_{1/r}^{\infty} e^{it\omega} \omega^{\frac{n}{2}-\gamma} J_{\frac{n-2}{2}}(r\omega) d\omega
&=\frac{1}{\sqrt{2\pi}} e^{-i\frac{\pi(n-1)}{4}} r^{-1/2}\int_{1/r}^{\infty} e^{i(t+r)\omega} \omega^{\frac{n-1}{2}-\gamma} d\omega\\
\nonumber&+\frac{1}{\sqrt{2\pi}} e^{i\frac{\pi(n-1)}{4}} r^{-1/2}\int_{1/r}^{\infty} e^{i(t-r)\omega} \omega^{\frac{n-1}{2}-\gamma}d\omega\\
&+ \int_{1/r}^{\infty} e^{it\omega} \omega^{\frac{n}{2}-\gamma}  R_{\frac{n-2}{2}}(r\omega) d\omega.
\end{align}
To bound the first integral in the right side of \eqref{cance}, we first denote $A=t+r$ and may assume $A>0$ without loss of generality.
We also set $\alpha=\gamma-\frac{n-1}{2}$.
Changing variables $A\omega\rightarrow\omega$ and
using the fact that $e^{i\omega}$ is periodic with period $2\pi$, we now see that
\begin{align}\label{thm}
\nonumber
\int_{1/r}^{\infty} e^{i(t+r)\omega} \omega^{\frac{n-1}{2}-\gamma} d\omega
&=A^{\alpha-1} \sum_{k=0}^{\infty} \int_{\frac{A}{r}+2\pi k}^{\frac{A}{r}+2\pi(k+1)} \omega^{-\alpha} e^{i\omega} d\omega \\
\nonumber
&=A^{\alpha-1}\sum_{k=0}^{\infty} \bigg[\int_{\frac{A}{r}+2\pi k}^{\frac{A}{r}+2\pi k+\pi} \omega^{-\alpha} e^{i\omega} d\omega
+\int_{\frac{A}{r}+2\pi k+\pi}^{\frac{A}{r}+2\pi(k+1)} \omega^{-\alpha} e^{i\omega} d\omega \bigg]\\
&= A^{\alpha-1}\int_{A/r}^{\infty} \big[ \omega^{-\alpha} - (\omega + \pi)^{-\alpha}\big]e^{i\omega} d\omega,
\end{align}
and note that for $\alpha>0$
\begin{align*}
\nonumber
\omega^{-\alpha} - (\omega + \pi)^{-\alpha} & = (\omega+\pi)^{-\alpha} \bigg[\Big(1+\frac{\pi}{\omega}\Big)^{\alpha} -1 \bigg] \\
&\lesssim \begin{cases}
(\omega+\pi)^{-\alpha} (\frac{\pi}{\omega})^{\alpha} \quad \text{if} \quad 0 <\omega \leq \pi,\\
(\omega+\pi)^{-\alpha}(\frac{\pi}{\omega}) \quad \text{if} \quad \omega > \pi
\end{cases}
\end{align*}
where we used the mean value theorem when $\omega>\pi$.
The first integral is now bounded as
\begin{align}
\nonumber
\bigg|\int_{1/r}^{\infty} e^{i(t+r)\omega} \omega^{\frac{n-1}{2}-\gamma} d\omega\bigg|
&\lesssim A^{\alpha-1}\bigg(\int_{A/r}^{\pi} (\omega+\pi)^{-\alpha} \omega^{-\alpha} d\omega
+ \int_{\pi}^{\infty} (\omega +\pi)^{-\alpha}\omega^{-1} d\omega \bigg)\\
\nonumber
&\lesssim A^{\alpha-1}\bigg(\int_{A/r}^{\pi} \omega^{-\alpha} d\omega + \int_{\pi}^{\infty} \omega^{-\alpha-1} d\omega\bigg) \\
&\lesssim A^{\alpha-1}
\label{upII_1}
\end{align}
if $0<\alpha<1$ which is equivalent to $\frac{n-1}2<\gamma<\frac{n+1}2$.
A similar argument with $A=t-r$ gives the same bound for the second integral in the right side of \eqref{cance}.
On the other hand, the last integral in \eqref{cance} is bounded by \eqref{err} as
\begin{align}
\nonumber
\bigg|\int_{1/r}^{\infty} e^{it\omega} \omega^{\frac{n}{2}-\gamma}  R_{\frac{n-2}{2}}(r\omega) d\omega\bigg|
& \leq \int_{1/r}^{\infty} \omega^{\frac{n}{2}-\gamma}  (r\omega)^{-3/2} d\omega\\
\nonumber
&\lesssim r^{-3/2}\int_{1/r}^{\infty} \omega^{\frac{n}{2}-\gamma-\frac{3}{2}} d\omega \\
\label{II_3}
&\lesssim r^{-\frac{n+2}{2}+\gamma}
\end{align}
if $\gamma>\frac{n-1}{2}$.
Combining \eqref{I}, \eqref{upII_1} and \eqref{II_3}, we therefore get
\begin{align} \label{qaz}
\nonumber
|K_{\gamma}(x,t)|&\lesssim r^{-\frac{n-2}{2}}\Big(r^{\frac{n-2}{2}}r^{-n+\gamma}+r^{-1/2}|A|^{\alpha-1}+r^{-\frac{n+2}{2}+\gamma}\Big)\\
 \nonumber&\lesssim |x|^{-n+\gamma}+|x|^{-\frac{n-1}{2}} \big| t\pm|x| \big|^{-\frac{n+1}{2}+\gamma}\\
 &\lesssim\begin{cases}
|x|^{-n+\gamma}\quad\text{if}\quad|x|\leq|t|/2,\\
|x|^{-\frac{n-1}{2}} \big| t\pm|x| \big|^{-\frac{n+1}{2}+\gamma}\quad\text{if}\quad|x|\geq|t|/2,
\end{cases}
\end{align}
if $\frac{n-1}2<\gamma<\frac{n+1}2$.

Alternatively as before (see \eqref{cance}), we now insert \eqref{asyBe} into the first integral in the right side of \eqref{ibp2} in the region $r\omega>1$ to see
\begin{align}\label{cance3}
\nonumber\int_{1/r}^{\infty} \frac{e^{it\omega}}{it} \omega^{\frac{n-2}{2}-\gamma} J_{\frac{n-2}{2}}(r\omega) d\omega
&=\frac{1}{it\sqrt{2\pi}} e^{-i\frac{\pi(n-1)}{4}} r^{-1/2}\int_{1/r}^{\infty} e^{i(t+r)\omega} \omega^{\frac{n-3}{2}-\gamma} d\omega\\
\nonumber&+\frac{1}{it\sqrt{2\pi}} e^{i\frac{\pi(n-1)}{4}} r^{-1/2}\int_{1/r}^{\infty} e^{i(t-r)\omega} \omega^{\frac{n-3}{2}-\gamma}d\omega\\
&+ \frac1{it}\int_{1/r}^{\infty} e^{it\omega} \omega^{\frac{n-2}{2}-\gamma}  R_{\frac{n-2}{2}}(r\omega) d\omega.
\end{align}
To estimate the first integral in the right side of \eqref{cance3}, we write it as
\begin{align*}
\int_{1/r}^{\infty} e^{i(t+r)\omega} \omega^{\frac{n-3}{2}-\gamma} d\omega
= A^{\alpha}\int_{A/r}^{\infty} \big[ \omega^{-(\alpha+1)} - (\omega + \pi)^{-(\alpha+1)}\big]e^{i\omega} d\omega
\end{align*}
with $A=t+r$ (see \eqref{thm}).
By the mean value theorem we then bound
\begin{align}\label{edn}
\nonumber\bigg|\int_{1/r}^{\infty} e^{i(t+r)\omega} \omega^{\frac{n-3}{2}-\gamma} d\omega\bigg|
&\lesssim A^{\alpha}\int_{A/r}^{\infty} \big|\omega^{-(\alpha+1)} - (\omega + \pi)^{-(\alpha+1)}\big| d\omega\\
\nonumber&\lesssim A^{\alpha}\int_{A/r}^{\infty} \omega^{-(\alpha+2)} d\omega\\
&\lesssim A^{-1}r^{\alpha+1}.
\end{align}
A similar argument with $A=t-r$ gives the same bound for the second integral in the right side of \eqref{cance3}.
On the other hand, the last integral in \eqref{cance3} is bounded by \eqref{err} as
\begin{align}
\nonumber
\bigg| \int_{1/r}^{\infty} e^{it\omega} \omega^{\frac{n-2}{2}-\gamma}  R_{\frac{n-2}{2}}(r\omega) d\omega \bigg|
\nonumber
&\lesssim  \int_{1/r}^{\infty} \omega^{\frac{n-2}{2}-\gamma} (r\omega)^{-3/2} d\omega \\
\nonumber
&= r^{-\frac{3}{2}} \int_{1/r}^{\infty} \omega^{\frac{n-5}{2}-\gamma} d\omega \\
\label{IV_3}
&\lesssim r^{-n/2+\gamma}
\end{align}
whenever $(n-3)/2<\gamma$.
Combining \eqref{cance3}, \eqref{edn} and \eqref{IV_3},
$$
\bigg|\int_{1/r}^{\infty} \frac{e^{it\omega}}{it} \omega^{\frac{n-2}{2}-\gamma} J_{\frac{n-2}{2}}(r\omega) d\omega\bigg|
\lesssim |t|^{-1}r^{-1/2}|A|^{-1}r^{\alpha+1}+|t|^{-1}r^{-n/2+\gamma}
$$
Hence by \eqref{osci2}
\begin{equation}\label{osci3}
\bigg|\int_0^{\infty} \frac{e^{it \omega}}{it} \omega^{\frac{n-2}{2}-\gamma}  J_{\frac{n-2}{2}}(r\omega) d\omega\bigg|
\lesssim |t|^{-1}r^{-n/2+\gamma}+|t|^{-1}r^{\alpha+1/2}|A|^{-1}
\end{equation}
if $\frac{n-1}2<\gamma<n-1$.
Applying the same argument to the second integral in the right side of \eqref{ibp2} gives
\begin{equation}\label{osci4}
\bigg|\int_0^{\infty} \frac{e^{it \omega}}{it} \omega^{\frac{n}{2}-\gamma}  J_{\frac{n}{2}}(r\omega) d\omega\bigg|
\lesssim |t|^{-1}r^{-n/2+\gamma-1}+|t|^{-1}|A|^{-1}r^{\alpha-1/2}
\end{equation}
if $\frac{n-1}2<\gamma<n+1$.
Consequently, by \eqref{K}, \eqref{ibp2}, \eqref{osci3} and \eqref{osci4},
\begin{align} \label{uio}
\nonumber|K_{\gamma}(x,t)|
&\lesssim |t|^{-1}|x|^{-n+\gamma+1}+|t|^{-1}|x|^{-n+\gamma+2}| t\pm|x||^{-1}\\
&\lesssim\begin{cases}
|t|^{-1}|x|^{-n+\gamma+1}\quad\text{if}\quad|x|\leq|t|/2,\\
|t|^{-1}|x|^{-n+\gamma+2}| t\pm|x||^{-1}\quad\text{if}\quad|x|\geq|t|/2,
\end{cases}
\end{align}
 if $\frac{n-1}2<\gamma<n-1$.
Combining \eqref{qaz} and \eqref{uio}, if $\frac{n-1}2<\gamma<\min\{\frac{n+1}2,n-1\}$, we conclude
\begin{align*}
|K_{\gamma}(x,t)|
&\lesssim\min\big\{|x|^{-n+\gamma},|t|^{-1}|x|^{-n+\gamma+1}\big\}\\
&\lesssim|t|^{-1}|x|^{-n+\gamma+1}
\end{align*}
when $|x|\leq|t|/2$, and when $|x|\geq|t|/2$ we choose \eqref{qaz} this time, which is better than \eqref{uio} to obtain fixed-time estimates;
this is because it decays faster than \eqref{uio} as $|x|\rightarrow\infty$ while it is less singular as $|x|\rightarrow|t|$.
\end{proof}

\section{Time-decay estimates} \label{sec4}
Now we finish the proof of Proposition \ref{fix} by making use of the pointwise estimates just obtained in the previous section.
The assumption \eqref{ftc} is equivalent to assume that
\begin{equation} \label{ftcl}
\frac{2n}{n-\gamma}<r\le \widetilde r <\min\{\frac{4}{n-2\gamma+1},\frac{2n}{n-\gamma-1}\}
\end{equation}
when $\gamma\in(\frac{n}{2},\frac{n+1}{2})$,
and when $\gamma\in(\frac{n+1}{2},n-1)$
\begin{equation} \label{ftcl2}
\frac{2n}{n-\gamma}<r\le \widetilde r <\frac{2n}{n-\gamma-1},
\end{equation}
just by noting $\frac{4}{n-2\gamma+1}=\frac{2n}{n-\gamma-1}$ when $\gamma=\frac{n}2+\frac1{n-1}$.
Under these assumptions we shall prove the time-decay estimates \eqref{x}.

We first choose $\varphi(x) \in C_0^{\infty}(\mathbb{R}^n)$ supported on $\{x\in\mathbb{R}^n:|x|\le1\}$
to calculate
\begin{equation}\label{fixedes}
\|K_{\gamma}(\cdot,t)\|_{W(\frac{\widetilde r}{2},\frac{r}{2})_x}^{r/2}
=\int_{\mathbb{R}^n}  \|K_{\gamma}(\cdot,t) \tau_y \varphi(\cdot)\|_{L_x^{\widetilde r/2}}^{\frac{r}{2}} dy,
\end{equation}
and set
\begin{equation}\label{II}
I_t(|y|):=\int_{|y|-1 \le |x| \le |y|+1} |K_{\gamma}(x,t)|^{{\widetilde r}/2} dx.
\end{equation}
By the size of $\textrm{supp}\,\varphi$, $\|K_{\gamma}(\cdot,t) \tau_y  \varphi(\cdot) \|_{L_x^{\widetilde r/2}}^{\widetilde r/2}$
is then calculated as
\begin{equation*}
\int_{|x-y| \le 1} |K_{\gamma}(x,t)|^{{\widetilde r}/2} dx
\lesssim
\begin{cases}
I_t(|y|)\quad\text{if} \quad |y|\leq1,\\
|y|^{-(n-1)}I_t(|y|)\quad\text{if} \quad |y|\ge1
\end{cases}
\end{equation*}
because the region $\{x\in\mathbb{R}^n:|y|-1 \le |x| \le |y|+1\}$ when $|y| \ge 1$
contains balls with radius $1$ as many as a constant multiple of $|y|^{n-1}$.
Finally, we calculate \eqref{fixedes} as
\begin{align}\label{secint0}
\nonumber\|K_{\gamma}(\cdot,t)\|_{W(\frac{\widetilde r}{2},\frac{r}{2})_x}^{r/2}
&\lesssim\int_{|y| \le 1} I_t(|y|)^{\frac{r}{\widetilde{r}}}dy
+\int_{|y|\ge1}   \big(|y|^{-(n-1)}I_t(|y|)\big)^{\frac{r}{\widetilde{r}}}dy\\
&:=A+B.
\end{align}

\subsection{The case $\gamma\in(\frac{n+1}2,n-1)$}\label{subsec4.1}
In this case we recall from \eqref{kerbd0} that
\begin{equation}\label{upbd0}
| K_{\gamma} (x,t) | \lesssim
\begin{cases}
|t|^{-1}|x|^{-n+1+\gamma} \quad {\textit{if}} \quad |x| \leq |t|/2, \\
|x|^{-n+\gamma}  \quad {\textit{if}} \quad |x| \geq |t|/2.
\end{cases}
\end{equation}

\subsubsection{Estimates for $ I_t(|y|)$}
We now estimate $ I_t(|y|)$ dividing cases into $|y| \le |t|/2-1$, $|t|/2-1\le|y| \le |t|/2+1$ and $|y| \ge |t|/2+1$
by considering both the size of $\textrm{supp}\,\varphi$ and the different behavior of \eqref{upbd0} near $|x|=|t|/2$.

\

\noindent \textit{(a) The case $|y| \le |t|/2-1$ (so $|t|\ge2$).}
In this case the integral \eqref{II} simply boils down to a single integral of the form $\int_a^b\rho^{\alpha-1}d\rho$ ($a\ge0$) with
\begin{equation}\label{alp}
\alpha=\alpha(n,\gamma,\widetilde{r}):=\frac{\widetilde r}{2}(-n+\gamma+1) +n>0,
\end{equation}
which is equivalent to the assumption
\begin{equation}\label{qwd}
\widetilde{r}<2n/(n-\gamma-1)
\end{equation}
in \eqref{ftcl2},
after using the first one only in \eqref{upbd0} and then the polar coordinates.
Indeed,

\begin{itemize}
\item when $|y| \le 1$;
\begin{equation}\label{a0}
I_t(|y|)
\lesssim|t|^{-\frac{\widetilde r}{2}} \int_0^{|y|+1} \rho^{\alpha-1} d\rho
\lesssim |t|^{-\frac{\widetilde r}{2}},
\end{equation}
\item when $|y| \ge 1$;
\begin{equation}\label{a00}
I_t(|y|)
\lesssim
|t|^{-\frac{\widetilde r}{2}} \int_{|y|-1}^{|y|+1} \rho^{\alpha-1} d\rho
\lesssim
|t|^{-\frac{\widetilde r}{2}} (|y|\pm1)^{\alpha-1}
\end{equation}
where we sort $+$ and $-$ from $\pm$ when $\alpha\geq1$ and $0<\alpha<1$, respectively.
\end{itemize}

\

\noindent \textit{(b) The case $|t|/2-1\le|y| \le |t|/2+1$.}
Additionally using the fact that $-\frac{\widetilde{r}}2+\alpha-1<0$
which is already satisfied by the assumption \eqref{ftcl2},
we similarly obtain
\begin{itemize}
\item when $|y| \le 1$;
\begin{align} \label{b0}
\nonumber
I_t(|y|)&\lesssim |t|^{-\frac{\widetilde r}{2}} \int_{0}^{|t|/2} \rho^{\alpha-1} d\rho+\int_{|t|/2}^{|y|+1} \rho^{-\frac{\widetilde{r}}2+\alpha-1}d\rho\\
&\lesssim|t|^{-\frac{\widetilde r}{2}+\alpha}+|t|^{-\frac{\widetilde{r}}2+\alpha-1},
\end{align}
\item when $|y| \ge 1$;
\begin{align} \label{b00}
\nonumber
I_t(|y|)&\lesssim |t|^{-\frac{\widetilde r}{2}} \int_{|y|-1}^{|t|/2} \rho^{\alpha-1} d\rho+\int_{|t|/2}^{|y|+1} \rho^{-\frac{\widetilde{r}}2+\alpha-1}d\rho\\
&\lesssim
\begin{cases}
|t|^{-\frac{\widetilde r}{2}+\alpha-1}\quad\text{if}\quad\alpha\ge1,\\
|t|^{-\frac{\widetilde r}{2}}(|y|-1)^{\alpha-1}+|t|^{-\frac{\widetilde r}{2}+\alpha-1}\quad\text{if}\quad0<\alpha<1.
\end{cases}
\end{align}
\end{itemize}

\

\noindent \textit{(c) The case $|y| \ge |t|/2+1$.}

\begin{equation} \label{c0}
I_t(|y|)\lesssim \int_{|y|-1}^{|y|+1} \rho^{-\frac{\widetilde{r}}2+\alpha-1}d\rho
\lesssim(|y|-1)^{-\frac{\widetilde{r}}2+\alpha-1}.
\end{equation}

\subsubsection{Putting things together}\label{subsubsec4.1.2}
We now estimate the first part $A$ in \eqref{secint0}
combining \eqref{a0} and \eqref{b0}, as follows:
\begin{itemize}
\item when $|t|\ge4$;
\begin{equation*}
A
\lesssim\int_{|y| \le 1} \eqref{a0}^{\frac{r}{\widetilde{r}}}dy
\lesssim|t|^{-\frac{r}{2}},
\end{equation*}
\item when $2\le|t|\le4$;
\begin{align*}
A
&\lesssim\int_{|y| \le \frac{|t|}2-1} \eqref{a0}^{\frac{r}{\widetilde{r}}}dy
+\int_{\frac{|t|}2-1\le|y| \le 1} \eqref{b0}^{\frac{r}{\widetilde{r}}}dy\\
&\lesssim|t|^{-\frac{r}{2}}+|t|^{-\frac{r}2+\frac{r}{\widetilde{r}}(\alpha-1)},
\end{align*}
\item when $|t|\le2$;
\begin{equation*}
A
\lesssim\int_{|y| \le 1} \eqref{b0}^{\frac{r}{\widetilde{r}}}dy
\lesssim|t|^{-\frac{r}2+\frac{r}{\widetilde{r}}(\alpha-1)}.
\end{equation*}
\end{itemize}

Next we estimate the second part $B$ in \eqref{secint0}
using \eqref{a00}, \eqref{b00} and \eqref{c0}. With the notation $d\mu=|y|^{-\frac{r}{\widetilde{r}}(n-1)}dy$, one can see that
\begin{itemize}
\item when $|t|\ge4$;
\begin{align*}
B
&\lesssim\int_{1\le|y| \le \frac{|t|}2-1} \eqref{a00}^{\frac{r}{\widetilde{r}}}d\mu
+\int_{\frac{|t|}2-1\le|y| \le \frac{|t|}2+1} \eqref{b00}^{\frac{r}{\widetilde{r}}}d\mu
+\int_{|y| \ge \frac{|t|}2+1} \eqref{c0}^{\frac{r}{\widetilde{r}}}d\mu\\
&\lesssim|t|^{-\frac{r}2+\frac{r}{\widetilde{r}}(\alpha-n)+n},
\end{align*}
\item when $|t|\le4$;
\begin{align*}
B
&\lesssim\int_{1\le|y| \le \frac{|t|}2+1} \eqref{b00}^{\frac{r}{\widetilde{r}}}d\mu
+\int_{|y| \ge \frac{|t|}2+1} \eqref{c0}^{\frac{r}{\widetilde{r}}}d\mu\\
&\lesssim|t|^{-\frac{r}2+\frac{r}{\widetilde{r}}(\alpha-1)}.
\end{align*}
\end{itemize}
It is the most delicate case to bound the integral on the region $|y|\ge|t|/2+1$.
It is only different from the other cases which can be done by a simple computation without additional assumptions,
and needs more explanation;
let $\kappa=\kappa(n,r,\widetilde{r}):=-\frac{r}{\widetilde r}(n-1)+n-1$.
Changing variables $\rho\rightarrow\rho+1$ and then applying the binomial theorem to
$(\rho+1)^{n}$, we estimate it when $|t|\geq4$ as
\begin{align}\label{4}
\int_{|t|/2+1}^\infty (\rho-1)^{\frac{r}{\widetilde{r}}(-\frac{\widetilde{r}}{2}+\alpha-1)}\rho^{\kappa} d\rho
&= \int_{|t|/2}^\infty \rho^{\frac{r}{\widetilde{r}}(-\frac{\widetilde{r}}{2}+\alpha-1)}(\rho+1)^{\kappa} d\rho \\
\nonumber&\lesssim \int_{|t|/2}^{\infty} \rho^{\frac{r}{\widetilde{r}}(\alpha-n)-\frac{r}2-1}(\rho+1)^{n} d\rho \\
\nonumber&= \sum_{m=0}^{n}C_{n,m}\int_{|t|/2}^{\infty}  \rho^{\frac{r}{\widetilde{r}}(\alpha-n)-\frac{r}2-1+m}d\rho \\
\nonumber&\sim |t|^{\frac{r}{\widetilde{r}}(\alpha-n)-\frac{r}2}\sum_{m=0}^{n}C_{n,m} |t|^{m}  \\
\nonumber&\lesssim |t|^{\frac{r}{\widetilde{r}}(\alpha-n)-\frac{r}2+n}
\end{align}
provided $\frac{r}{\widetilde{r}}(\alpha-n)-\frac{r}2+m=\frac{r}{2}(-n+\gamma)+m<0$ for all $0 \le m \le n$
which follows from the assumption
\begin{equation}\label{condd}
r>2n/(n-\gamma)
\end{equation}
in \eqref{ftcl2}.
Here, $C_{n,m}$ denotes the binomial coefficients.
When $|t|\leq4$ we estimate \eqref{4} in a different way as
$$
 \int_{|t|/2}^\infty \rho^{\frac{r}{\widetilde{r}}(-\frac{\widetilde{r}}{2}+\alpha-1)}(\rho+1)^{\kappa} d\rho
=\int_{|t|/2}^{2}\rho^{\frac{r}{\widetilde{r}}(-\frac{\widetilde{r}}{2}+\alpha-1)}(\rho+1)^{\kappa} d\rho
+\int_{2}^\infty \rho^{\frac{r}{\widetilde{r}}(-\frac{\widetilde{r}}{2}+\alpha-1)}(\rho+1)^{\kappa} d\rho.
$$
The second integral in the right side is bounded as above;
\begin{align*}
\int_{2}^\infty \rho^{\frac{r}{\widetilde{r}}(-\frac{\widetilde{r}}{2}+\alpha-1)}(\rho+1)^{\kappa} d\rho
&\lesssim  \sum_{m=0}^{n-1}C_{n,m}\int_{2}^{\infty}  \rho^{\frac{r}{\widetilde{r}}(\alpha-n)-\frac{r}2+m}d\rho\\
&\lesssim  \sum_{m=0}^{n-1}C_{n,m} \lesssim 1.
\end{align*}
On the other hand, the first one is bounded as
$$
\int_{|t|/2}^{2} \rho^{\frac{r}{\widetilde{r}}(-\frac{\widetilde{r}}{2}+\alpha-1)}(\rho+1)^{\kappa}d\rho
\lesssim \int_{|t|/2}^{2}\rho^{\frac{r}{\widetilde{r}}(-\frac{\widetilde{r}}{2}+\alpha-1)}d\rho
\lesssim |t|^{\frac{r}{\widetilde{r}}(-\frac{\widetilde{r}}{2}+\alpha-1)}
$$
provided
\begin{equation}\label{conds}
\kappa\geq0 \quad(\text{i.e.},\, r\leq\widetilde{r})
\end{equation}
and
\begin{equation*}
\frac{r}{\widetilde{r}}(-\frac{\widetilde{r}}{2}+\alpha-1)=\frac{r}{2}(-n+\gamma)+\frac{r}{\widetilde r}(n-1)<0
\end{equation*}
which are already satisfied by the assumption \eqref{ftcl2}.

Combining \eqref{secint0} with these bounds for $A,B$, we conclude that
\begin{equation*}
\|K_{\gamma}(\cdot,t)\|_{W(\frac{\widetilde r}{2},\frac{r}{2})_x}^{r/2}
\lesssim
\begin{cases}
|t|^{-\frac{r}2+\frac{r}{\widetilde{r}}(\alpha-n)+n}\quad\text{if}\quad|t|\geq1,\\
|t|^{\frac{r}{\widetilde{r}}(-\frac{\widetilde{r}}{2}+\alpha-1)}\quad\text{if}\quad|t|\leq1,
\end{cases}
\end{equation*}
under the same assumptions \eqref{qwd}, \eqref{condd} and \eqref{conds} as in \eqref{ftcl2}.
The proof is now complete by substituting \eqref{alp} for $\alpha$ here.

\subsection{The case $\gamma\in(\frac{n}2,\frac{n+1}2)$}
This case is handled in the same way as well. So we shall omit the details.
Recall from \eqref{kerbd} that
\begin{equation}\label{upbd}
| K_{\gamma} (x,t) | \lesssim
\begin{cases}
|t|^{-1}|x|^{-n+1+\gamma} \quad {\textit{if}} \quad |x| \leq |t|/2, \\
|x|^{-\frac{n-1}{2}} \big||x|-|t|\big|^{-\frac{n}{2}-\frac{1}{2}+\gamma} \quad {\textit{if}} \quad |x| \geq |t|/2.
\end{cases}
\end{equation}
This time, the region of $|y|$ is split into more than those in the previous case
because \eqref{upbd} behaves differently near $|x|=|t|$ as well as $|x|=|t|/2$:
$$
|y| \le \frac{|t|}{2}-1, \quad \frac{|t|}{2}-1 \le |y| \le \frac{|t|}{2}+1,\quad|y| \ge \frac{|t|}2+1,
$$
$$
\frac{|t|}{2}+1 \le |y| \le |t|-1,\quad |t|-1 \le |y| \le |t|+1, \quad |y| \ge |t|+1.
$$
To calculate $A$ and $B$ in \eqref{secint0} as before, we need to estimate \eqref{II} on these regions.
It ultimately boils down to estimations for two integrals of the forms
$\int_a^b\rho^{\alpha-1}d\rho$ ($a\ge0$)
and $\int_c^d\rho^\lambda\big|\rho-|t|\big|^{\beta}d\rho$ ($c>0$)
with the same $\alpha>0$ given in \eqref{alp},
$$\lambda=\lambda(n,\widetilde{r}):=-\frac{\widetilde r}{2} \cdot \frac{n-1}{2} +n-1<0\quad\text{and}\quad
\beta=\beta(n,\gamma,\widetilde{r}):=\frac{\widetilde r}{2}(-\frac{n}{2}-\frac{1}{2}+\gamma)<0,$$
after using \eqref{upbd} and the polar coordinates.
Here the conditions $\alpha>0$ and $\lambda,\beta<0$ are satisfied by the assumption \eqref{ftcl}.
The integral $\int_a^b\rho^{\alpha-1}d\rho$ is handled in the same way as before,
and therefore we only need to show how to handle the other integral $\int_c^d\rho^\lambda\big|\rho-|t|\big|^{\beta}d\rho$;
note first that
$$\int_c^d\rho^\lambda\big|\rho-|t|\big|^{\beta}d\rho\lesssim c^\lambda\int_{c-|t|}^{d-|t|}|\rho|^{\beta}d\rho$$
since $\lambda<0$ and $c>0$, and then estimate this by dividing cases as follows:
\begin{itemize}
\item when $0<c-|t|$;
$$\int_{c-|t|}^{d-|t|}|\rho|^{\beta}d\rho\lesssim(c-|t|)^\beta(d-c),$$
\item when $c-|t|<0<d-|t|$;
\begin{align*}
\int_{c-|t|}^{d-|t|}|\rho|^{\beta}d\rho
&\lesssim\int_{c-|t|}^{0}(-\rho)^{\beta}d\rho+\int_{0}^{d-|t|}\rho^{\beta}d\rho\\
&\lesssim(|t|-c)^{\beta+1}+(d-|t|)^{\beta+1}
\end{align*}
provided $\beta+1>0$ equivalent to the remaining assumption $\widetilde{r}<4/(n-2\gamma+1)$ in \eqref{ftcl},
\item when $d-|t|<0$;
$$\int_{c-|t|}^{d-|t|}|\rho|^{\beta}d\rho=\int_{c-|t|}^{d-|t|}(-\rho)^{\beta}d\rho\lesssim(|t|-d)^\beta(d-c).$$
\end{itemize}

\section{Retarded estimates}\label{sec5}
Now we obtain the retarded estimates \eqref{reta} in Theorem \ref{thm01} based on the time-decay estimates in Proposition \ref{fix}.
By \eqref{ker} the desired estimates are rephrased as
\begin{equation} \label{reta1}
	\bigg\| \int_0^t (K_{\gamma}(\cdot,t-s) * F(\cdot,s))(x) ds \bigg\|_{W(\widetilde q ,q)_t W(\widetilde r,r)_x} \lesssim \|F\|_{W(\widetilde q_1' ,q_1')_t W(\widetilde r_1',r_1')_x}.
\end{equation}
By Minkowski's inequality and the convolution relation \eqref{y-ineq}, it follows that
\begin{align} \label{bfhls3}
	\nonumber \bigg\| \int_0^t (K_\gamma (\cdot,t-s)* &F(\cdot,s))(x) ds \bigg\|_{W(\widetilde q ,q)_t W(\widetilde r ,r)_x} \\
	&\le \bigg\| \int_0^t \|K_{\gamma}(\cdot,t-s)\|_{W(\widetilde r_0 ,r_0)_x} \|F(\cdot,s)\|_{W(\widetilde r_1',r_1')_x} ds\bigg\|_{W(\widetilde q ,q)_t}
\end{align}
where $1/{\widetilde r_0}:=1/{\widetilde r}+1/{\widetilde r_1}$ and $1/{r_0}:=1/r+1/r_1$.
Using \eqref{hls} with $\alpha=1/q+1/q_1$ and $p=q_1'$ and the usual Young's inequality, the convolution relation gives
\begin{equation*}
	W(L^{\widetilde q_0}, L^{\frac1\alpha, \infty})_t * W(\widetilde q_1 ', q_1')_t \subset W(\widetilde q, q )_t
\end{equation*}
if
\begin{equation*}
	0\le\frac{1}{\widetilde q_0}:=\frac{1}{\widetilde q}+\frac{1}{\widetilde q_1}\leq1, \quad  0<\frac{1}{q}+\frac{1}{q_1}<1 \quad (q\neq\infty).
\end{equation*}
Hence we get
\begin{align} \label{bfhls4}
	\nonumber
	\bigg\| \int_0^t \|K_{\gamma}(\cdot,t-s)&\|_{W(\widetilde r_0 ,r_0)_x}\|F(\cdot,s)\|_{W(\widetilde r_1',r_1')_x} ds\bigg\|_{W(\widetilde q ,q)_t} \\
	&\qquad \lesssim
	\| K_{\gamma} \|_{{W(L^{\widetilde q_0},  L^{\frac1\alpha,\infty})_t}{W(\widetilde r_0 ,r_0)_x}}
	\|F\|_{W(\widetilde q_1' ,q_1')_t W(\widetilde r_1',r_1')_x}.
\end{align}
Combining \eqref{bfhls3} and \eqref{bfhls4}, we now obtain the desired estimate \eqref{reta1} if
\begin{equation}\label{t1}
	\| K_{\gamma} \|_{{W(L^{\widetilde q_0},  L^{\frac1\alpha,\infty})_t}{W(\widetilde r_0 ,r_0)_x}}< \infty
\end{equation}
with$\frac{1}{\widetilde q_0}=\frac{1}{\widetilde q}+\frac{1}{\widetilde q_1}$, $\alpha =\frac{1}{q}+\frac{1}{q_1}$,
$\frac{1}{\widetilde r_0} = \frac{1}{\widetilde r} + \frac{1}{\widetilde r_1}$ and $\frac{1}{r_0} = \frac{1}{r} + \frac{1}{r_1}$
 for $(\widetilde q, \widetilde r)$, $(q,r)$, $(\widetilde q_1, \widetilde r_1)$ and $(q_1,r_1)$ given as in Theorem \ref{thm01}.

One can show \eqref{t1} obviously in the same way as \eqref{t}.
So we omit the details;
set $h(t)= \|K_{\gamma}(\cdot,t)\|_{W(\widetilde r_0,r_0)_x}$
and choose $\varphi(t) \in C_0^{\infty}(\mathbb{R})$ supported on $\{t\in\mathbb{R}:|t|\le1\}$.
Using the time-decay estimates \eqref{x} with $\widetilde r = 2\widetilde r_0$ and $r=2 r_0$,
\begin{equation} \label{local_t1}
	\|h\tau_k\varphi\|_{L_t^{\widetilde q_0}} \lesssim
	\begin{cases}
		1 \quad\textit{if}\quad |k| \leq 2,\\
		(|k|-1)^{-(n-\gamma - \frac{n}{r_0})} \quad\textit{if}\quad |k| \geq 2
	\end{cases}
\end{equation}
as before (see \eqref{local_t}) under the conditions in Theorem \ref{thm01}.
By \eqref{local_t1}, $\|h\tau_k\varphi\|_{L_t^{\widetilde q_0}}$  belongs to $L^{1/\alpha, \infty}_k$
since $\alpha = n-\gamma-\frac{n}{r_0}$ from the second condition in \eqref{ic}.
This finally implies
\begin{equation*}
	\| h \|_{W(L^{\widetilde q_0}, L^{\frac1\alpha,\infty})_t} < \infty
\end{equation*}
as desired.

\section{Local well-posedness}\label{sec6}
This final section is devoted to proving Theorem \ref{thmthm}.
By Duhamel's principle, we first write the solution to \eqref{eq} as
\begin{equation}\label{zxc}
	\Phi(u)=\cos (t\sqrt{-\Delta})f + \frac{\sin (t\sqrt{-\Delta})}{\sqrt{-\Delta}} g+ \int_0^t \frac{\sin((t-s)\sqrt{-\Delta})}{\sqrt{-\Delta}} F_k(u)(\cdot,s)ds.
\end{equation}
Then we will make use of the homogeneous and retarded estimates to each of the terms in \eqref{zxc} to show that
$\Phi$ defines a contraction map on
\begin{equation*}
	X(T,M) = \big\{u \in W(\widetilde q,q)_t(I;W(\widetilde r , r)_x(\mathbb{R}^3)): \|u\|_{W(\widetilde q,q)_t(I;W(\widetilde r , r)_x(\mathbb{R}^3))} \leq M\big\}
\end{equation*}
for appropriate values of $T, M>0$.
Here, $I:=[0,T]$ and $(\widetilde q, \widetilde r), (q, r)$ are given as in Theorem \ref{thmthm}.
To begin with, we need the following homogeneous estimates with low Sobolev norms and
the inhomogeneous estimates exactly suit to the Duhamel term in \eqref{zxc}:

\begin{cor} \label{corr}
Let $n=3$ and $0<\sigma <1$. Assume that
\begin{equation} \label{h1}
	0<\frac{1}{q} < \frac{1}{\widetilde q} \le \frac{\sigma}{2} \quad \text{and} \quad
	\frac{1-\sigma}{2} < \frac{1}{\widetilde r} \leq \frac{1}{r} < \frac{3-2\sigma}{6}.
\end{equation}
Then we have
\begin{equation} \label{T1}
	\|e^{it\sqrt{-\Delta}}f\|_{W(\widetilde q , q)_t W(\widetilde r,r)_x} \lesssim \|f\|_{\dot H^{\sigma}}
\end{equation}
if
\begin{equation} \label{h2}
	\frac{1}{\widetilde q} + \frac{2}{\widetilde r} > 1-\frac{\sigma}{2} \quad \text{and} \quad \frac{1}{q} + \frac{3}{r} = \frac{3}{2}-\sigma.
\end{equation}
\end{cor}

\begin{cor}\label{cor}
	Let $n=3$.
	Assume that
	\begin{equation} \label{inhoc}
		0<\frac{1}{q}+\frac{1}{q_1} < \frac{1}{\widetilde q} + \frac{1}{\widetilde q_1} \leq \frac{1}{2} \quad(q\neq\infty)
\quad \text{and} \quad \frac{1}{2} < \frac{1}{\widetilde r} + \frac{1}{\widetilde r_1} \leq \frac{1}{r} + \frac{1}{r_1} < \frac{2}{3}.
	\end{equation}
	Then we have
	\begin{equation} \label{inho}
		\bigg\|\int_0^t e^{i(t-s)\sqrt{-\Delta}}|\nabla|^{-1}F(\cdot,s)ds\bigg\|_{W(\widetilde q,q)_t W(\widetilde r , r)_x}
		\lesssim \|F\|_{W(\widetilde q_1',q_1')_t W(\widetilde r_1' , r_1')_x}
	\end{equation}
	if
	\begin{equation} \label{inhoc1}
		\frac{1}{\widetilde q} + \frac{1}{\widetilde q_1} + \frac{2}{\widetilde r} + \frac{2}{\widetilde r_1} > \frac{3}{2} \quad \text{and} \quad \frac{1}{q} +\frac{1}{q_1} + \frac{3}{r} + \frac{3}{r_1} = 2.
	\end{equation}
\end{cor}

\begin{rem} \label{r}
If $(\widetilde{q},q,\widetilde{r},r)$ and $(\widetilde{q}_1,q_1,\widetilde{r}_1,r_1)$ are given as in Corollary \ref{corr} with
$\sigma$ and $1-\sigma$, respectively, then they satisfy the conditions \eqref{inhoc} and \eqref{inhoc1}.
This fact is usefully used to prove \eqref{data} in Theorem \ref{thmthm}.
\end{rem}

Assuming for the moment these corollaries which will be derived in the rest of this section from our main estimates,
we first show that $\Phi(u)\in X$ for $u\in X$.
For this, we apply the homogeneous estimates \eqref{T1} to the homogeneous terms in \eqref{zxc} to get
\begin{equation*}
	\bigg\|\cos (t\sqrt{-\Delta})f + \frac{\sin (t\sqrt{-\Delta})}{\sqrt{-\Delta}} g\bigg\|_{W(\widetilde q,q)_t(I;W(\widetilde r , r)_x)}
	\leq C \|f\|_{\dot H^{\sigma}} + C\|g\|_{\dot H^{\sigma-1}}
\end{equation*}
under the conditions in Corollary \ref{corr}.
On the other hand, we apply \eqref{inho} to the Duhamel term in \eqref{zxc} to see
\begin{equation*}
	\bigg\|\int_0^t \frac{\sin((t-s)\sqrt{-\Delta})}{\sqrt{-\Delta}} F_k(u)(\cdot,s)ds\bigg\|_{W(\widetilde q,q)_t(I;W(\widetilde r , r)_x)}
	\leq C\|F_k(u)\|_{W(\widetilde q_1',q_1')_t(I;W(\widetilde r_1' , r_1')_x)}
\end{equation*}
under the conditions in Corollary \ref{cor}.
By using the inclusion relation \eqref{inclusion} and  H\"older's inequality together with the assumption \eqref{cF},
the right side here is bounded as
\begin{align} \label{wd2}
	\nonumber
	C\|F_k(u)\|_{W(\widetilde q_1',q_1')_t(I;W(\widetilde r_1' , r_2')_x)} &\leq C\|\chi_I\|_{W(\widetilde q_0,q_0)_t} \||u|^k\|_{W(\frac{\widetilde q}{k},\frac{q}{k})_t W(\frac{\widetilde r}{k}, \frac{r}{k})_x} \\
	&\leq C T^{\frac{1}{\widetilde{q}_0}}M^{k}
\end{align}
if $T<1$
provided
\begin{equation} \label{conn}
\frac1{r_2}\leq\frac1{r_1},\quad \frac1{\widetilde q_0}>0,
\end{equation}
\begin{equation} \label{con1}
\frac{1}{\widetilde q_0} +\frac{k}{\widetilde q} + \frac{1}{\widetilde q_1}= 1, \quad \frac{1}{q_0} +\frac{k}{q} + \frac{1}{q_1}= 1, \quad \frac{k}{\widetilde r} + \frac{1}{\widetilde r_1} = 1, \quad
	\frac{k}{r}+\frac{1}{r_2}=1.
\end{equation}
Hence, if we fix $M=2C(\|f\|_{\dot H^{\sigma}}+\|g\|_{\dot H^{\sigma-1}})$ and take $T<1$ such that
\begin{equation} \label{M}
	CT^{\frac{1}{\widetilde{q}_0}}M^{k-1} \le \frac{1}{2},
\end{equation}	
we get
\begin{equation}\label{qds}
	\|\Phi(u)\|_{W(\widetilde q,q)_t(I;W(\widetilde r , r)_x)} \leq  M
\end{equation}
for $(\widetilde{q},\widetilde{r})$ and $(q,r)$ given as in Theorem \ref{thmthm}.
Indeed, we first eliminate  $r_2,\widetilde q_0,q_0$ in \eqref{conn} and \eqref{con1} to see
\begin{equation} \label{con}
0 \le \frac{k}{\widetilde q} + \frac{1}{\widetilde q_1} <1,
\quad 0 \le \frac{k}{q} + \frac{1}{q_1} \le 1,
\quad \frac{k}{\widetilde r} + \frac{1}{\widetilde r_1} = 1,
\quad 1-\frac{k}{r} \le \frac{1}{r_1}.
\end{equation}
We then eliminate the remaining redundant pairs $\widetilde q_1, q_1 , \widetilde {r}_1,r_1$
in \eqref{con}, \eqref{inhoc} and \eqref{inhoc1};
we substitute $1/\widetilde{r}_1 = 1-k/\widetilde r$ into \eqref{inhoc} and \eqref{inhoc1} to eliminate $\widetilde{r}_1$,
and then make each lower bound of $1/r_1,1/\widetilde {q}_1,1/q_1$ less than all the upper bounds thereof in turn to eliminate them,
to reduce  \eqref{con}, \eqref{inhoc} and \eqref{inhoc1} to
\begin{equation}\label{rdd}
\begin{aligned}
 0\le\frac{1}{\widetilde q},\, \frac{1}{q} < \frac{1}{k-1} \quad(q\neq\infty),\quad
 \frac{1}{3(k-1)}&< \frac{1}{\widetilde r} < \frac{1}{2(k-1)}, \quad \frac{1}{3(k-1)} < \frac{1}{r},\\
 \frac{1}{\widetilde q} + \frac{2}{\widetilde r} &< \frac{3}{2(k-1)}.
\end{aligned}
\end{equation}
In summary, all the requirements on $\widetilde{q},q,\widetilde{r},r$ for which \eqref{qds} holds are given by
\eqref{rdd}, \eqref{h1} and \eqref{h2},
in which one can eliminate $\sigma,k$ to boil down to \eqref{qaz2}.
On the other hand, if we eliminate $\widetilde{q},q,\widetilde{r},r$ in the requirements under $0<\sigma\le 1/2$ (so $1<k\leq3$),
we arrive at
\begin{equation*}
	\max \{0, 1-\frac{1}{k-1} \} < \sigma \leq \frac{1}{2}
\end{equation*}
equivalent to $0<\sigma\le 1/2$ and $1<k< k(\sigma)$ as in Theorem \ref{thmthm}.

Next we show that $\Phi$ is a contraction on $X$.
Note first from the assumption \eqref{cF} that
\begin{align*}
	|F_k(u)-F_k(v)| &= \bigg| \int_0^1 \frac{d}{d\eta} F_k(\eta u + (1-\eta)v) d\eta \bigg| \\
	&= \bigg|\int_0^1 (u-v)\cdot F_k' (\eta u +(1-\eta)v) d\eta \bigg| \\
	&\lesssim |u-v|(|u|+|v|)^{k-1}.
\end{align*}
Using the same argument as above, we then see that
\begin{align*}
	\|&\Phi(u)-\Phi(v)\|_{W(\widetilde q,q)_t(I;W(\widetilde r , r)_x)}\\
&\leq C \|F_k(u)-F_k(v)\|_{W(\widetilde q_1',q_1')_t(I;W(\widetilde r_1' , r_2')_x)}\\
&  \leq
	C\|\chi_{I}\|_{W(\widetilde q_0,q_0)_t} \|u-v\|_{W(\widetilde q,q)_t W(\widetilde r, r)_x} \|(|u|+|v|)^{k-1}\|_{W(\frac{\widetilde q}{k-1},\frac{q}{k-1})_t W(\frac{\widetilde r}{k-1}, \frac{r}{k-1})_x} \\
	& \leq C T^{\frac1{\widetilde{q}_0}} M^{k-1} \|u-v\|_{W(\widetilde q,q)_t(I;W(\widetilde r , r)_x)}
\end{align*}
for the same $(\widetilde{q},\widetilde{r})$ and $(q,r)$ as above.
Hence $\Phi$ is a contraction on $X$ since we are taking $T$ and $M$ so that \eqref{M} holds.
Now by the contraction mapping principle, there exists a unique solution
\begin{equation*}
	u \in W(\widetilde q, q)_t([0,T]; W(\widetilde r,r)_x(\mathbb{R}^3))
\end{equation*}
for given initial data $(f,g)\in(\dot{H}^{\sigma},\dot{H}^{\sigma-1})$.

It remains to show \eqref{data}. We first show $u\in C_t([0,T];\dot{H}^{\sigma})$.
Since $e^{it\sqrt{-\Delta}}$ is an isometry in $L^2$, we first see
\begin{equation*}
\sup_{t\in I}\|\cos (t\sqrt{-\Delta})f\|_{\dot H^{\sigma}} +	\sup_{t\in I}\Big\|\frac{\sin (t\sqrt{-\Delta})}{\sqrt{-\Delta}}g\Big\|_{\dot H^{\sigma}} \leq C( \|f\|_{\dot H^{\sigma}} + \|g\|_{\dot H^{\sigma-1}}),
\end{equation*}
while
\begin{align}\label{kkl}
	\nonumber
	\sup_{t \in I} \bigg\|\int_0^t \frac{\sin{(t-s)\sqrt{-\Delta}}}{\sqrt{-\Delta}} F_k(u) ds\bigg\|_{\dot{H}^{\sigma}} &\lesssim\sup_{t \in I} \bigg\|\int_0^t e^{i(t-s)\sqrt{-\Delta}}|\nabla|^{-(1-\sigma)} F_k(u) ds\bigg\|_{L^2} \\
&		\lesssim  \|F_k(u)\|_{W(\widetilde q_1',q_1')_t(I;W(\widetilde r_1',r_1')_x)}
\end{align}
with $(\widetilde{q}_1,q_1,\widetilde{r}_1,r_1)$ given as in Corollary \ref{corr} with $1-\sigma$.
For \eqref{kkl} we also used
the following adjoint form of \eqref{T1},
\begin{equation*}
	\bigg\| \int_{-\infty}^{\infty} e^{i(t-s)\sqrt{-\Delta}}|\nabla|^{-\sigma} F(\cdot,s)ds\bigg\|_{L^2} \lesssim \|F\|_{W(\widetilde q_1' , q_1')_t W(\widetilde r_1' ,r_1')_x}
\end{equation*}
(see \eqref{dudu}).
Recalling Remark \ref{r} and using the same argument as in \eqref{wd2}, we then get
\begin{equation} \label{CC1}
	\sup_{t \in I} \|u\|_{\dot{H}^{\sigma}} \lesssim \|f\|_{\dot{H}^{\sigma}} + \|g\|_{\dot{H}^{\sigma-1}}.
\end{equation}
The other assertion $u\in C_t^1([0,T];\dot{H}^{\sigma-1})$ can be also proved similarly.
Using
$$\partial_t u = \cos (t\sqrt{-\Delta})g - \sin (t\sqrt{-\Delta})|\nabla|f+ \int_0^t {\cos((t-s)\sqrt{-\Delta})} F_k(u)(\cdot,s)ds$$
one can indeed see that
\begin{equation} \label{CC2}
	\sup_{t \in I} \|u\|_{\dot{H}^{\sigma-1}} \lesssim \|f\|_{\dot{H}^{\sigma}} + \|g\|_{\dot{H}^{\sigma-1}}.
\end{equation}
Continuous dependence on the data is
similarly included in the above arguments.
This completes the proof.

\begin{proof}[Proof of Corollaries]
To obtain \eqref{T1}, we apply the complex interpolation \eqref{inter} between \eqref{classi} with $(q,r,\sigma)=(\infty,2,0)$
and \eqref{T} with $(\widetilde q, \widetilde r) = (\widetilde q_1,\widetilde r_1)$ and $(q,r,\sigma)=(q_1,r_1,\sigma_1)$ arbitrarily near $\sigma_1=1$ to obtain \eqref{T1} for
\begin{equation} \label{hh}
	\frac{1}{\widetilde q} = \frac{\theta}{\widetilde q_1},
	\quad \frac{1}{q} = \frac{\theta}{{q_1}},
	\quad \frac{1}{\widetilde r} =\frac{\theta}{\widetilde r_1}+\frac{1-\theta}{2},
	\quad \frac{1}{r} =\frac{\theta}{{r_1}}+\frac{1-\theta}{2}, \quad 0<\sigma=\theta<1.
\end{equation}
Since $\widetilde q_1 , q_1, \widetilde r_1, r_1$ in \eqref{hh} are satisfying \eqref{ass} and \eqref{c1} with $n=3$,
the conditions on $\widetilde q , q, \widetilde r, r$ in Corollary \ref{corr} follows.

To obtain \eqref{inho}, we similarly make use of the complex interpolation between \eqref{reta} and
\begin{equation} \label{in1}
\bigg\|\int_0^t e^{i(t-s)\sqrt{-\Delta}}|\nabla|^{-2\sigma}F(\cdot,s)ds \bigg\|_{W(q,q)_t W(r,r)_x} \lesssim \|F\|_{W(q',q')_tW(r',r')_x}
\end{equation}
where $q>q'$ and $(q,r,\sigma)$ is given as in \eqref{waveadmi}.
This estimate is easily derived from \eqref{classi} using the $TT^\ast$ argument and the Christ-Kiselev lemma.
We shall also use the following interpolation space identities.

\begin{lem} [\cite{BL}] \label{inS}
	Let $0<\theta<1$, $1\le r_0, r_1<\infty$ and $\sigma_0, \sigma_1 \in \mathbb{R}$.
	Then
	\begin{equation*}
		(\dot H_{r_0}^{\sigma_0}, \dot H_{r_1}^{\sigma_1})_{[\theta]}= \dot H_r^{\sigma},
	\end{equation*}
	if ${1}/{r}={\theta}/{r_0} + {(1-\theta)}/{r_1}$ and $\sigma=\theta \sigma_0 + (1-\theta)\sigma_1$ with $\sigma_0\neq \sigma_1$.
Here, $\dot H_r^{\sigma}=|\nabla|^{-\sigma}L^r$ denotes the homogeneous Sobolev space.
\end{lem}

Indeed, by applying the complex interpolation
between \eqref{reta} with $\widetilde r = r$ and \eqref{in1}
with $(q,r,\sigma)$ arbitrarily near $(\infty,2,0)$,
we first get
	\begin{equation} \label{in2}
		\bigg\|\int_0^t e^{i(t-s)\sqrt{-\Delta}}|\nabla|^{-1}F(\cdot,s)ds \bigg\|_{W(\widetilde q,q)_t W(r,r)_x} \lesssim \|F\|_{W(\widetilde q_1', q_1')_t W(\widetilde r_1',r_1')_x}
	\end{equation}
under \eqref{inhoc} and \eqref{inhoc1} with $\widetilde r = r$.
Again by interpolating between \eqref{reta} with $\widetilde r_1 = r_1$ and \eqref{in1} with $(q,r,\sigma)$ arbitrarily near $(\infty,2,0)$,
we also get
	\begin{equation} \label{in3}
		\bigg\|\int_0^t e^{i(t-s)\sqrt{-\Delta}}|\nabla|^{-1}F(\cdot,s)ds \bigg\|_{W(\widetilde q, q)_t W(\widetilde r,r)_x} \lesssim \|F\|_{W(\widetilde q_1',q_1')_t W(r_1',r_1')_x}
	\end{equation}
under \eqref{inhoc} and \eqref{inhoc1} with $\widetilde r_1 = r_1$.
Finally by interpolating between these two estimates \eqref{in2} and \eqref{in3},
we obtain \eqref{inho} as desired.	
\end{proof}


\end{document}